\documentclass[12pt]{amsart}
\calclayout
\usepackage{amssymb}
\usepackage{amsmath}
\usepackage{graphics}
\usepackage{epsfig, color}
\newcommand\R{\mathbb R}

\newcommand\C{\mathcal C}
\newcommand\K{\mathcal K}
\renewcommand\O{\mathcal O}
\newcommand\T{\mathcal T}
\newcommand\E{\mathcal E}
\newcommand\M{\mathcal M}

\newcommand\N{\mathcal N}
\newcommand\V{\mathcal V}

\renewcommand\H{\mathcal H}

\renewcommand\div{\operatorname{div}}

\newcommand\eps{\operatorname{\epsilon}}

\newcommand\x{\times}
\renewcommand\t{\tilde}
\newcommand\lbra{[\![}
\newcommand\rbra{]\!]}
\newcommand\lbrac{\,[\!\!\!\{}
\newcommand\rbrac{\}\!\!\!]\,}

\renewcommand\ll{|\kern-2pt|\kern-2pt|}

\newcommand\e{^\epsilon}

\numberwithin{equation}{section}
\theoremstyle{plain}
\newtheorem{thm}{Theorem}

\newtheorem{lem}[thm]{Lemma}

\numberwithin{thm}{section}

\theoremstyle{remark}

\def\underput#1#2#3{
\mathchoice
{\vtop{\ialign{##\crcr\hfil$#2$\vrule width0pt height0pt depth#3\hfil\crcr
\noalign{\nointerlineskip}\hfil$\scriptstyle#1$\hfil\crcr}}}
{\vtop{\ialign{##\crcr\hfil$#2$\vrule width0pt height0pt depth#3\hfil\crcr
\noalign{\nointerlineskip}\hfil$\scriptstyle#1$\hfil\crcr}}}
{\vtop{\ialign{##\crcr\hfil$\scriptstyle#2$\vrule width0pt height0pt
depth#3\hfil\crcr
\noalign{\nointerlineskip}\hfil$\scriptscriptstyle#1$\hfil\crcr}}}
{\vtop{\ialign{##\crcr\hfil$\scriptscriptstyle#2$\vrule width0pt height0pt
depth#3\hfil\crcr
\noalign{\nointerlineskip}\hfil$\scriptscriptstyle#1$\hfil\crcr}}}}

\def\stack#1#2#3{\rlap{#1}\lower#3\hbox{#2}}

\def\twiddlespace{1.2truept}

\def\dtwiddle{\displaystyle\sim}
\def\ttwiddle{\textstyle\sim}
\def\stwiddle{\scriptstyle\sim}
\def\sstwiddle{\scriptscriptstyle\sim}

\def\doubledtwiddle{\stack{$\dtwiddle$}{$\dtwiddle$}{\twiddlespace}}
\def\doublettwiddle{\stack{$\ttwiddle$}{$\ttwiddle$}{\twiddlespace}}
\def\doublestwiddle{\stack{$\stwiddle$}{$\stwiddle$}{\twiddlespace}}
\def\doublesstwiddle{\stack{$\sstwiddle$}{$\sstwiddle$}{\twiddlespace}}

\def\tripledtwiddle{\stack{$\dtwiddle$}{$\doubledtwiddle$}{\twiddlespace}}
\def\triplettwiddle{\stack{$\ttwiddle$}{$\doublettwiddle$}{\twiddlespace}}
\def\triplestwiddle{\stack{$\stwiddle$}{$\doublestwiddle$}{\twiddlespace}}
\def\triplesstwiddle{\stack{$\sstwiddle$}{$\doublesstwiddle$}{\twiddlespace}}

\def\quadrupledtwiddle{\stack{$\dtwiddle$}{$\tripledtwiddle$}{\twiddlespace}}
\def\quadruplettwiddle{\stack{$\ttwiddle$}{$\triplettwiddle$}{\twiddlespace}}
\def\quadruplestwiddle{\stack{$\stwiddle$}{$\triplestwiddle$}{\twiddlespace}}
\def\quadruplesstwiddle{\stack{$\sstwiddle$}{$\triplesstwiddle$}{\twiddlespace}}\def\quadru
pletwiddle{{\mathchoice{\quadrupledtwiddle}{\quadruplettwiddle}%

{\quadruplestwiddle}{\quadruplesstwiddle}}}

\def\strikedist{3pt}

\def\dstrike{\vrule width7pt height0pt depth.4pt}
\def\tstrike{\vrule width7pt height0pt depth.4pt}
\def\sstrike{\hbox{\vrule width5pt height0pt depth.4pt}}
\def\ssstrike{\hbox{\vrule width3pt height0pt depth.4pt}}
\def\strike{{\mathchoice{\dstrike}{\tstrike}{\sstrike}{\ssstrike}}}

\def\ub#1{\underput\strike{#1}{\strikedist}}

\catcode`\!=11
\newcommand\!a{{\boldsymbol a}}
\newcommand\!b{{\boldsymbol b}}
\newcommand\!c{{\boldsymbol c}}
\newcommand\!d{{\boldsymbol d}}
\newcommand\!e{{\boldsymbol e}}
\newcommand\!f{{\boldsymbol f}}
\newcommand\!g{{\boldsymbol g}}
\newcommand\!h{{\boldsymbol h}}
\newcommand\!i{{\boldsymbol i}}
\newcommand\!j{{\boldsymbol j}}
\newcommand\!k{{\boldsymbol k}}
\newcommand\!l{{\boldsymbol l}}
\newcommand\!m{{\boldsymbol m}}
\newcommand\!n{{\boldsymbol n}}
\newcommand\!o{{\boldsymbol o}}
\newcommand\!p{{\boldsymbol p}}
\newcommand\!q{{\boldsymbol q}}
\newcommand\!r{{\boldsymbol r}}
\newcommand\!s{{\boldsymbol s}}
\newcommand\!t{{\boldsymbol t}}
\newcommand\!u{{\boldsymbol u}}
\newcommand\!v{{\boldsymbol v}}
\newcommand\!w{{\boldsymbol w}}
\newcommand\!x{{\boldsymbol x}}
\newcommand\!y{{\boldsymbol y}}
\newcommand\!z{{\boldsymbol z}}
\newcommand\!A{{\boldsymbol A}}
\newcommand\!B{{\boldsymbol B}}
\newcommand\!C{{\boldsymbol C}}
\newcommand\!D{{\boldsymbol D}}
\newcommand\!E{{\boldsymbol E}}
\newcommand\!F{{\boldsymbol F}}
\newcommand\!G{{\boldsymbol G}}
\newcommand\!H{{\boldsymbol H}}
\newcommand\!I{{\boldsymbol I}}
\newcommand\!J{{\boldsymbol J}}
\newcommand\!K{{\boldsymbol K}}
\newcommand\!L{{\boldsymbol L}}
\newcommand\!M{{\boldsymbol M}}
\newcommand\!N{{\boldsymbol N}}
\newcommand\!O{{\boldsymbol O}}
\newcommand\!P{{\boldsymbol P}}
\newcommand\!Q{{\boldsymbol Q}}
\newcommand\!R{{\boldsymbol R}}
\newcommand\!S{{\boldsymbol S}}
\newcommand\!T{{\boldsymbol T}}
\newcommand\!U{{\boldsymbol U}}
\newcommand\!V{{\boldsymbol V}}
\newcommand\!W{{\boldsymbol W}}
\newcommand\!X{{\boldsymbol X}}
\newcommand\!Y{{\boldsymbol Y}}
\newcommand\!Z{{\boldsymbol Z}}
\newcommand\!alpha{{\boldsymbol\alpha}}
\newcommand\!beta{{\boldsymbol\beta}}
\newcommand\!gamma{{\boldsymbol\gamma}}
\newcommand\!delta{{\boldsymbol\delta}}
\newcommand\!epsilon{{\boldsymbol\epsilon}}
\newcommand\!zeta{{\boldsymbol\zeta}}
\newcommand\!eta{{\boldsymbol\eta}}
\newcommand\!theta{{\boldsymbol\theta}}
\newcommand\!iota{{\boldsymbol\iota}}
\newcommand\!kappa{{\boldsymbol\kappa}}
\newcommand\!lambda{{\boldsymbol\lambda}}
\newcommand\!mu{{\boldsymbol\mu}}
\newcommand\!nu{{\boldsymbol\nu}}
\newcommand\!xi{{\boldsymbol\xi}}
\newcommand\!pi{{\boldsymbol\pi}}
\newcommand\!rho{{\boldsymbol\rho}}
\newcommand\!sigma{{\boldsymbol\sigma}}
\newcommand\!tau{{\boldsymbol\tau}}
\newcommand\!upsilon{{\boldsymbol\upsilon}}
\newcommand\!phi{{\boldsymbol\phi}}
\newcommand\!chi{{\boldsymbol\chi}}
\newcommand\!psi{{\boldsymbol\psi}}
\newcommand\!omega{{\boldsymbol\omega}}
\newcommand\!varepsilon{{\boldsymbol\varepsilon}}
\newcommand\!vartheta{{\boldsymbol\vartheta}}
\newcommand\!varpi{{\boldsymbol\varpi}}
\newcommand\!varrho{{\boldsymbol\varrho}}
\newcommand\!varsigma{{\boldsymbol\varsigma}}
\newcommand\!varphi{{\boldsymbol\varphi}}
\newcommand\!Gamma{{\boldsymbol\Gamma}}
\newcommand\!Delta{{\boldsymbol\Delta}}
\newcommand\!Theta{{\boldsymbol\Theta}}
\newcommand\!Lambda{{\boldsymbol\Lambda}}
\newcommand\!Xi{{\boldsymbol\Xi}}
\newcommand\!Pi{{\boldsymbol\Pi}}
\newcommand\!Sigma{{\boldsymbol\Omega\eigma}}
\newcommand\!Upsilon{{\boldsymbol\Upsilon}}
\newcommand\!Phi{{\boldsymbol\Phi}}
\newcommand\!Psi{{\boldsymbol\Psi}}
\newcommand\!Omega{{\boldsymbol\Omega}}

\begin{document}

\title [Finite element for Naghdi shell]
{A discontinuous Galerkin method\\ for the Naghdi shell model}


\author{Sheng Zhang}

\thanks{Department of Mathematics, Wayne State University, Detroit, MI 48202 (\texttt{szhang@wayne.edu})}

\begin{abstract}
We propose a mixed discontinuous Galerkin method for the bending problem of Naghdi shell, and present an 
analysis for its accuracy.
The error estimate shows that when components of the curvature tensor and Christoffel symbols 
are piecewise linear functions,
the finite element method has the optimal order of accuracy, which is uniform with respect to the shell
thickness. Generally, the error estimate shows how the accuracy  is affected by the shell geometry and thickness.
It suggests that  to achieve optimal rate of convergence, the triangulation should be
properly refined in regions where the shell geometry changes dramatically.
These are the results 
for a balanced method in which the primary displacement components and rotation components 
are approximated by discontinuous
piecewise quadratic  polynomials, while components of the scaled membrane stress tensor and shear stress vector
are approximated by continuous piecewise linear functions. On elements that have edges on the free boundary 
of the shell, finite element space for displacement components needs to be 
enriched slightly, for stability purpose. 
Results on higher order finite elements are also included.

\vspace{12pt}

\end{abstract}
\maketitle


\section{Introduction}

We propose a mixed finite element method for the Naghdi shell model and present an analysis for its accuracy.
In the method, the midsurface displacement and normal fiber rotation  are
approximated by discontinuous piecewise polynomials, while the scaled membrane stress and transverse shear stress 
are approximated by
continuous piecewise polynomials.
This is a discontinuous Galerkin (DG) method in terms of the primary variables of the Naghdi shell model.
The finite elements for various variables  form a balanced combination in the sense that
except for some minor displacement enrichments required by stability on the free edge of the shell,
every degree of freedom contributes to the accuracy of the finite element solution.
DG method provides a more general approach
and offers  more flexibilities in choosing finite element spaces and degree of freedoms.
It is believed to have a potential to help resolve some difficult problems in numerical computation of elastic shells \cite{AF-RM-DG, A-RM-DG, Feng}.
In this paper, we show that DG method indeed has advantages in reducing the troublesome
membrane and shear locking in computation of shell bending problems.
We prove an error estimate showing that when the geometry of a shell satisfies certain conditions the method yields
a finite element solution that has the
optimal order of accuracy that could be achieved by the best approximation from the finite element functions.
Thus it is free of membrane/shear locking.
When such condition is not satisfied, the estimate shows how the accuracy is affected
by the geometrical coefficients and how to adjust the finite element mesh to accommodate the curved shell mid-surface such that
the finite element solution
achieves the optimal order of accuracy.
Particularly, the estimate suggests that some refinements for the finite element mesh should be done
where a shell changes geometry abruptly.

We consider a thin shell of thickness $2\eps$.
Its middle surface $\t\Omega\subset\R^3$ is the image of a two-dimensional coordinate domain $\Omega\subset\R^2$ through
the parameterization mapping $\!phi:\Omega\to\t\Omega$. This mapping furnishes the curvilinear coordinates on the surface $\t\Omega$.
Subject to loading forces and boundary conditions, the shell would be deformed to a stressed state.
The Naghdi shell model uses displacement of the shell mid-surface and rotation of normal fibers as the primary variables.
The tangential displacement is represented by its covariant components $u\e_{\alpha}$ ($\alpha\in\{1,2\}$), normal displacement is a scalar $w\e$, 
and the rotation is a vector with covariant components $\theta\e_\alpha$.
(The superscript $\eps$ indicates dependence on the shell thickness.)
To deal with membrane and transverse shear locking, we also introduce
the transverse shear stress vector and symmetric membrane stress tensor both scaled by multiplying the factor $\eps^{-2}$  as independent variables,
which are given in terms their  contravariant components $\xi^{\eps\alpha}$ and $\M^{\eps\alpha\beta}$ ($\alpha, \beta\in\{1,2\}$).
All the ten functions are two-variable functions defined on $\Omega$.
For a bending dominated shell problem, under a suitable scaling on the loading force, these functions converge to finite limits when $\eps\to 0$.
This justifies our choice of approximating them as independent variables.
For a curved shell deformation to be bending dominated,  the shell needs to have a portion of its boundary free, 
or subject to
force conditions. It is known that a totally clamped or simply supported  elliptic, parabolic, or hyperbolic
shell does not allow bending dominated behavior \cite{CiarletIII}.
We assume the shell boundary is divided into three parts, on which the shell is clamped, simply supported, and free of displacement constraint,
respectively, and the free part is not empty.

We assume that the coordinate domain $\Omega$ is a polygon. On $\Omega$, we introduce a triangulation $\T_h$
that is shape regular but not necessarily quasi-uniform. The shape regularity of a triangle is defined as the ratio of the diameter
of its smallest circumscribed circle and the diameter of its largest inscribed circle. The shape regularity of a triangulation is the maximum of
shape regularities of all its triangular elements. When we say a $\T_h$ is shape regular we mean that the shape regularity
of $\T_h$ is bounded by an absolute constant $\K$.  Shape regular meshes allow local refinements, and thus have the potential to
more efficiently resolve the ever increasing singularities in solutions of the shell model.
We use $\T_h$ to
denote the set of all the (open) triangular
elements, and let $\Omega_h=\cup_{\tau\in\T_h}\tau$.
We use $h_\tau$ to denote the diameter of the element $\tau$.
We analyze a particular combination of finite elements for various variables.
We use totally discontinuous piecewise quadratic polynomials to approximate the
displacement components and rotation components  $u\e_\alpha$, $w\e$, and $\theta\e_\alpha$,  
and use continuous piecewise linear functions to
approximate components of the scaled membrane stress tensor $\M^{\eps\alpha\beta}$ and scaled shear stress vector $\xi^{\eps\alpha}$ .
If an element $\tau$ has one edge that lies on the free boundary of the shell, we need to enrich the space of quadratic
polynomials for the displacement by adding two cubic polynomials. If an element has two edges on the free 
boundary, we need to use the the full cubic polynomials for the displacements. 
The finite element model
yields an approximation $(\theta^h_\alpha, u^h_{\alpha}, w^h)$, $(\M^{h \alpha\beta}, \xi^{h\alpha})$,
and we have the error estimate that 
\begin{multline}\label{estimate1}
\|(\!theta\e-\!theta^h, \!u\e-\!u^h, w\e-w^h)\|_{\H_h}\\
\le C
\left[1+\eps^{-1}
\max_{\tau\in\T_h}h^3_\tau\left(\sum_{\alpha,\beta,\lambda=1,2}|\Gamma^{\lambda}_{\alpha\beta}|_{2,\infty,\tau}+
\sum_{\alpha,\beta=1,2}|b_{\alpha\beta}|_{2,\infty,\tau}+\sum_{\alpha,\beta=1,2}|b^\beta_\alpha|_{2,\infty,\tau}\right)
\right]
\\
\left\{
\sum_{\tau\in\T_h}h^4_{\tau}\left[\sum_{\alpha=1,2}\left(\|\theta\e_\alpha\|^2_{3,\tau}+\|u\e_\alpha\|^2_{3,\tau}\right)+\|w\e\|^2_{3,\tau}+
\sum_{\alpha, \beta=1,2}\|\M^{\eps\alpha\beta}\|^2_{2,\tau}+\sum_{\alpha=1,2}\|\xi^{\eps\alpha}\|^2_{2,\tau}
\right]\right\}^{1/2}.
\end{multline}
Here, $C$ is a constant that could be dependent on the shape regularity $\K$  of $\T_h$ and the shell mid-surface, but otherwise
it is independent of the finite element mesh, the shell thickness, and the shell model solution.
For a subdomain $\tau\subset\Omega$, we use 
$\|\cdot\|_{k,\tau}$ and $|\cdot|_{k,\tau}$ to denote the norm and semi norm of the Sobolev space $H^k(\tau)$,
and use $\|\cdot\|_{k,\infty,\tau}$ and $|\cdot|_{k,\infty,\tau}$ to denote that of $W^{k,\infty}(\tau)$.
When $\tau=\Omega$, the space $H^k(\Omega)$ will be simply written as $H^k$.
The functions $\Gamma^{\lambda}_{\alpha\beta}$ are
the Christoffel symbols and $b_{\alpha\beta}$ and $b^\alpha_\beta$ the covariant and mixed components of curvature tensor of 
the parameterized shell middle surface $\t\Omega$. 
These will be called geometrical coefficients of the shell.
The left hand side norm is the  piecewise $H^1$ norm for the error $\theta\e_\alpha-\theta^h_\alpha$, $u\e_\alpha-u^h_{\alpha}$, and $w\e-w^h$,
plus penalties on discontinuity and violation of the essential boundary conditions by the finite element
approximation, see \eqref{Hh-norm} below.
It is noted that we have no estimate in for the errors $\!xi^{\eps\alpha}-\!xi^{h\alpha}$ and  $\M^{\eps\alpha\beta}-\M^{h\alpha\beta}$ in \eqref{estimate1},
while $\!xi^{\eps\alpha}$ and $\M^{\eps\alpha\beta}$ are  involved in the
right hand side, which usually has very strong internal and boundary layers. Some weaker estimate for this error will be given below.

The quantity in the first bracket in the right hand side of \eqref{estimate1}
is independent of the shell model solution. It, however, involves the geometrical coefficients,
the triangulation $\T_h$, and the shell thickness $\eps$.
If the geometric coefficients of the shell midsurface
are piecewise linear functions, then the quantity is completely
independent of $\eps$.
Generally, $\eps$ has some negative effect. To keep the quantity bounded,
the finite element mesh needs to be relatively fine
where the geometric coefficients has greater second order derivatives.
Where the shell is flat, the thickness $\eps$
does not impose much restriction on the mesh size.
In any case, the quantity in the first bracket is bounded if $h^3=\O(\eps)$, with $h$ being the maximum size of finite elements.
The finite element mesh, the shell shape, and its thickness together should satisfy a condition
such that the quantity in the first bracket is bounded.
It seems reasonable to say that the method reduces locking
quite significantly, which otherwise would amplify the error by a factor of the magnitude $\eps^{-1}$.

To assess the accuracy of the finite element solution,
we scale the loading forces on the shell by multiplying them with the factor $\eps^2$. (Such scaling will not affect relative errors
of numerical solutions.)
Then we have the limiting behaviors that
when $\eps\to 0$,
$\theta\e_\alpha\to\theta^0_\alpha$, $u\e_\alpha\to u^0_\alpha$,  and $w\e\to w^0$ in $H^1$, and $\xi^{\eps\alpha}$ and  $\M^{\eps\alpha\beta}$ 
converge to finite limits in a weaker norm.
The shell problem is bending dominated if and only if $(\theta^0_\alpha, u^0_\alpha, w^0)\neq 0$.
In this case, the smallness of the error in the left hand side of \eqref{estimate1} means small relative error
of the approximation of the primary variables, thus accuracy of the finite element model.
The asymptotic behaviors of $\theta\e_\alpha$, $u\e_\alpha$, $w\e$, $\xi^{\eps\alpha}$, and $\M^{\eps\alpha\beta}$ , in terms of convergence in strong or weak norms,
mean that they
tend to limiting functions in major  part of the domain, while
may exhibit boundary or internal layers that occur in slimmer and slimmer portions of the domain.
If the finite element functions are capable of
resolving such  singular layers,
the finite element solution would be accurate and free of membrane and shear locking.
It is noted that the quantity in the brace in the right hand side of \eqref{estimate1}
is the error estimate of the best approximations of $\theta\e_\alpha$, $u\e_\alpha$, $w\e$, $\xi^{\eps\alpha}$ and $\M^{\eps\alpha\beta}$
from their finite element functions in the piecewise  $H^1$-norm and $L^2$-norm, respectively.

If the limit $(\theta^0_\alpha, u^0_\alpha, w^0)$ is zero,
the shell deformation is not bending dominated. In this case
we do not have the accuracy of the finite element solution measured in the aforementioned relative error.
In computation, one would obtain finite element solutions that converge to zero  in the norm
in the left hand side of \eqref{estimate1} when $\eps\to 0$.
Our theory implies that
such smallness must not be due to numerical membrane or shear locking.
But rather, it indicates that the shell problem is not bending dominated, and
needs to be treated differently, in which case standard finite element methods could be better.
Whether a shell problem is bending dominated, membrane/shear  dominated, or intermediate
is determined by the shell shape, loading force, and boundary conditions \cite{CiarletIII, Bathe-book}.
Membrane/shear  locking is the most critical issue in bending dominated problems \cite{ABrezzi2}.

There is a huge literature on scientific computing of shell models.
Despite great success in numerical computation  in shell mechanics, the mathematical theory of numerical analysis is much less developed,
see the books \cite{Bernadou, CiarletIII, Hughes-book, Bathe-book} for reviews. 
There are several theories on locking free finite elements that are relevant to this paper.
In \cite{ABrezzi2}, a locking free estimate was established for Naghdi shell under the assumption that the geometrical coefficients  are piecewise constants.
In \cite{Suri}, similar result was proved for some higher order finite elements under the assumption that
the geometrical coefficients are higher order piecewise polynomials.
These papers did not say how the finite element accuracy would be affected had the assumptions on the geometrical coefficients  not been met.
In \cite{Bramble-Sun2}, a uniform accuracy of a finite element method was proved for Naghdi shell model
under a condition of the form $h^2\le \O(\eps)$, with some bubble functions introduced to enhance the stability.
Our result seems more general and our method seems simpler.
We have made an effort not to assume the finite mesh to be quasi-uniform.  This is important for
the shell model for which layers of singularities are very common in its solution and quasi-uniform meshes
are not practical.
The stability achieved in this paper are mainly due to the flexibility of discontinuous approximations. 
A theory for Koiter shell that is related to this paper can be found in \cite{Koiter-shell}.

The paper is organized as follows. In Section~\ref{SHELL} we recall the shell model in the standard variational form, 
and write it in a mixed form by introducing the scaled transverse shear stress vector and scaled membrane stress tensor as new variables. 
An asymptotic estimate on the model solution, and an equivalent estimate on the solution of the mixed model
are given in an abstract setting. The latter will  be used in analysis of the finite element model.
In Section~\ref{principle}, we introduce the finite element model that is consistent with the mixed 
form of the Naghdi shell model. The consistency is verified in the appendix. The appendix  also includes proofs for the abstract 
results.
In Section~\ref{KornOnShell}, we prove a discrete version of Korn's inequality that is suitable for the Naghdi shell.
This inequality plays a fundamental role in the error analysis. The error analysis is carried out in Section~\ref{ErrorAnalysis}.
In the last section, we briefly report  the results for higher order finite elements.

For a fixed $\eps$, the shell model solution $\theta\e_\alpha, u\e_\alpha, w\e$ will be assumed to have the $H^3$ regularity.
Of course, when $\eps\to 0$ these functions could  go to 
infinity in this norm.
Throughout the paper, $C$ is a constant that could be dependent on the geometrical coefficients of the shell mid-surface, 
the Lam\'e coefficients of the elastic material,  and shape regularity $\K$
of the triangulation  $\T_h$. It is otherwise independent of the triangulation and shell thickness. We shall simply say 
that the constant $C$ is independent of $\T_h$ and $\eps$.
For such a constant $C$, we use $A\lesssim B$ to denote $A\le CB$. If $A\lesssim B$ and $A\lesssim B$, we write $A\simeq B$.
Superscripts indicate contravariant components of vectors and tensors, and subscripts indicate covariant components.
Greek sub and super scripts, except $\eps$,  
take their values in $\{1,2\}$. Latin scripts take their values in $\{1,2,3\}$. Summation rules with respect to repeated sub and
super scripts will also be used. Vectors with covariant  components $u_\alpha$ or contravariant components $\xi^\alpha$ 
will be  represented by the bold face letter $\!u$ or $\!xi$, respectively. 
A tensor with components $\M^{\alpha\beta}$ will be simply called $\M$.


\section{The shell model}
\label{SHELL}
Let $\t\Omega\subset\R^3$
be the middle surface of a shell of thickness $2\eps$.
It is  the image of a domain
$\Omega\subset\R^2$ through a mapping $\!phi$.
The coordinates $x_\alpha\in\Omega$
then furnish the curvilinear coordinates on $\t\Omega$.
We assume that at any point on the surface,
along the coordinate lines,
the two tangential vectors
$\!a_{\alpha}={\partial\!phi}/{\partial x_{\alpha}}$
are linearly independent.
The unit vector
$\!a_3=(\!a_1\x\!a_2)/|\!a_1\x\!a_2|$ is normal to $\t\Omega$.
The triple $\!a_i$ furnishes the covariant basis on $\t\Omega$.
The contravariant basis
$\!a^i$ is defined by the relations
$\!a^{\alpha}\cdot\!a_{\beta}=\delta^{\alpha}_{\beta}$ and $\!a^3=\!a_3$,
in which $\delta^{\alpha}_{\beta}$ is the Kronecker delta.
It is obvious that $\!a^{\alpha}$ are also tangent to the surface.
The metric tensor has the covariant components
$a_{\alpha\beta}=\!a_{\alpha}\cdot\!a_{\beta}$,  the determinant of
which is denoted by $a$. The contravariant components
are given by
$a^{\alpha\beta}=\!a^{\alpha}\cdot\!a^{\beta}$.
The curvature tensor
has covariant components
$b_{\alpha\beta}=\!a_3\cdot\partial_{\beta}\!a_{\alpha}$, whose
mixed components are $b^{\alpha}_{\beta}=a^{\alpha\gamma}b_{\gamma\beta}$.
The symmetric tensor $c_{\alpha\beta}=b^\gamma_\alpha b_{\gamma\beta}$ is called the third 
fundamental form of the surface.
The Christoffel symbols
are defined by
$\Gamma^{\gamma}_{\alpha\beta}
=\!a^{\gamma}\cdot\partial_{\beta}\!a_{\alpha}$,
which are symmetric with respect to the subscripts. The derivative of a scalar is a covariant vector.
The covariant derivative of a vector or tensor is a higher order tensor.
The formulas below will be used in the following.
\begin{equation}\label{covariant-derivative}
\begin{gathered}
u_{\alpha|\beta}=\partial_{\beta}u_{\alpha}-\Gamma^{\gamma}_{\alpha\beta}
u_{\gamma},\quad
\eta^\alpha|_\beta=\partial_\beta\eta^\alpha+\Gamma^\alpha_{\beta\delta}\eta^\delta,\\
\sigma^{\alpha\beta}|_{\gamma}=\partial_{\gamma}\sigma^{\alpha\beta}
+\Gamma^{\alpha}_{\gamma\lambda}\sigma^{\lambda\beta}
+\Gamma^{\beta}_{\gamma\tau}\sigma^{\alpha\tau}.
\end{gathered}
\end{equation}
Product rules for differentiations, like
$(\sigma^{\alpha\lambda}u_{\lambda})|_{\beta}=
\sigma^{\alpha\lambda}|_{\beta}u_{\lambda}
+\sigma^{\alpha\lambda}u_{\lambda|\beta}$,
are valid. For more information see \cite{GZ}.

The mapping $\!phi$ is a one-to-one correspondence between $\Omega$ and $\t\Omega$. It maps a subdomain $\tau\subset\Omega$
to a subregion $\t\tau=\!phi(\tau)\subset\t\Omega$. 
A function $f$ defined on the shell middle surface will be identified with a function defined on $\Omega$ through the mapping $\!phi$ and denoted 
by the same notation. Thus $f(\!phi(x_\alpha))=f(x_\alpha)$. The integral over $\t\tau$ with respect to the surface area element 
is related to the double integral on $\tau$ by
\begin{equation*}
\int_{\t\tau}fd\t S=\int_\tau f\sqrt a dx_1dx_2.
\end{equation*}
We will ignore the area element $d\t S$ in the integral over the surface $\t\tau$, and 
simply write the left hand side integral as  $\displaystyle\int_{\t\tau}f$, and ignore the $dx_1dx_2$ in integral on the subdomain $\tau$, 
and write the right hand side integral as $\displaystyle\int_{\tau}f\sqrt a$.
The mapping $\!phi$ maps a curve $e\subset\overline\Omega$ to a curve $\t e=\!phi(e)$ contained in the closure of 
$\t\Omega$. Let $x_\alpha(s)$ be the arc length parameterization
of $e$, then $\!phi(x_\alpha(s))$ is a parameterization of $\t e$, but not in terms of the arc length of $\t e$. Let $\t s$ be the arc length 
parameter of $\t e$,  then the line integrals are related by
\begin{equation*}
\int_{\t e}fd\t s=\int_ef\sqrt{\sum_{\alpha,\beta=1, 2}a_{\alpha\beta}\frac{dx_\alpha}{ds}\frac{dx_\beta}{ds}}ds.
\end{equation*}
Similar to surface integrals, we will ignore the $d\t s$ in the left hand side line integral and the $ds$ in the right hand side line  integral.
Without further explanation, a tilde indicates operations on the curved surface $\t\Omega$, and no tilde means the operations are on the flat
domain $\Omega$. 
For any line element $e$, area element $\tau$, and function $f$ that 
make the following integrals meaningful, we have 
\begin{equation*}
\int_{\t\tau}|f|\simeq\int_\tau|f|,\quad
\int_{\t e}|f|\simeq\int_e|f|.
\end{equation*}
We need to repeatedly do integration by parts on the shell mid-surface, by using the Green's theorem on surfaces.
Let $\tau\subset\Omega$ be a subdomain, which is mapped to
the subregion $\tilde\tau\subset\t\Omega$ by $\!phi$.
Let $\!n=n_{\alpha}\!a^{\alpha}$ 
be the unit outward normal to the boundary $\partial\tilde\tau=\!phi(\partial\tau)$
which is tangent to the surface $\t\Omega$.
Let $\bar n_\alpha\!e^\alpha$ be the  unit outward normal vector to $\partial\tau$ in $\R^2$. Here $\!e^\alpha$ is the basis vector in $\R^2$.
The Green's theorem \cite{GZ} says that for a vector field $f^\alpha$,
\begin{equation}\label{Green}
\int_{\tilde\tau}f^{\alpha}|_{\alpha}=
\int_{\partial\tilde\tau}f^{\alpha}n_{\alpha}=
\int_{\partial\tau}f^{\alpha}\bar n_{\alpha}\sqrt a.
\end{equation}

\subsection{The Naghdi shell model}
The Naghdi shell model \cite{Naghdi, thesis} uses displacement $\!u, w$ of the shell mid-surface and normal fiber rotation $\!theta$ as the primary variables.
The bending strain, membrane strain, and transverse shear strain due to the deformation represented by 
such a set of primary variables are   
\begin{equation}\label{N-bending}
\rho_{\alpha\beta}(\!theta, \!u, w)=
\frac12(\theta_{\alpha|\beta}+\theta_{\beta|\alpha})-\frac12(b^\gamma_\alpha u_{\gamma|\beta}+b^\gamma_\beta u_{\gamma|\alpha})+c_{\alpha\beta}w,
\end{equation}
\begin{equation}\label{N-metric}
\gamma_{\alpha\beta}(\!u,w)=
\frac12(u_{\alpha|\beta}+u_{\beta|\alpha})
-b_{\alpha\beta}w,
\end{equation}
\begin{equation}\label{N-shear}
\tau_\alpha(\!theta, \!u, w)=\partial_\alpha w+b^\gamma_\alpha u_\gamma+\theta_\alpha.
\end{equation}

The loading forces on the shell body and upper and lower surfaces
enter the shell model as resultant loading forces per unit area  on the shell middle surface,
of which the tangential force density is
$p^{\alpha}\!a_{\alpha}$ and transverse force density $p^3\!a_3$.
Let the boundary $\partial\t\Omega$ be divided to $\partial^D\t\Omega\cup\partial^S\t\Omega\cup\partial^F\t\Omega$.
On $\partial^D\t\Omega$ the shell is clamped, on $\partial^S\t\Omega$ the shell is soft-simply supported, and
on $\partial^F\t\Omega$ the shell is free of displacement constraint and subject to force or moment  only.
(There are $32$ different ways to specify boundary conditions at any point on the shell boundary, of which we consider the three most typical.)
The shell model is
defined in the Hilbert space
\begin{multline}\label{N-space}
H=\{(\!phi, \!v, z)\in \!H^1\x\!H^1\x H^1;\  v_\alpha \text{ and } z \text{ are }0\  \text{on}\ \partial^D\Omega\cup\partial^S\Omega, \\
\text{ and  }\theta_\alpha \text{ is }0\
\text{on}\ \partial^D\Omega\}.
\end{multline}
The model determines a unique $(\!theta\e, \!u\e, w\e)\in H$
such that
\begin{multline}\label{N-model}
\frac13\int_{\t\Omega}
a^{\alpha\beta\lambda\gamma}\rho_{\lambda\gamma}(\!theta\e, \!u\e, w\e)
\rho_{\alpha\beta}
(\!phi, \!v, z)\\
+\eps^{-2}\int_{\t\Omega}
a^{\alpha\beta\lambda\gamma}\gamma_{\lambda\gamma}(\!u\e,w\e)
\gamma_{\alpha\beta}(\!v,z)+\kappa\mu\eps^{-2}\int_{\t\Omega}a^{\alpha\beta}\tau_\alpha(\!theta\e, \!u\e, w)\tau_\beta(\!phi, \!v, z)
\\
=
\int_{\t\Omega}
(p^{\alpha}v_{\alpha}+
p^3z)
+\int_{\partial^S\t\Omega}r^\alpha\phi_\alpha
+\int_{\partial^F\t\Omega}\left(q^\alpha v_\alpha+q^3z+r^\alpha\phi_\alpha\right)
\ \ \forall\
(\!phi, \!v,z) \in H.
\end{multline}
Here, $q^i$ and $r^\alpha$ are the force resultant and
moment resultant on the shell edge \cite{Naghdi}.
The factor $\kappa$ is a shear correction factor, often assumed to be $5/6$, which we think should be $1$ \cite{Naghdi-arch}.
The fourth order contravariant tensor
$a^{\alpha\beta\gamma\delta}$ is the elastic tensor of the shell,
defined by
\begin{equation*}
a^{\alpha\beta\gamma\delta}=\mu (a^{\alpha\gamma}a^{\beta\delta}+a^{\beta\gamma}a^{\alpha\delta})+
\frac{2\mu\lambda}{2\mu+\lambda}
a^{\alpha\beta}a^{\gamma\delta}.
\end{equation*}
Here, $\lambda$ and $\mu$ are the Lam\'e coefficients of the elastic material, which we assume to be constant.
The compliance tensor of the shell defines the inverse operator of the elastic tensor, given by
\begin{equation*}
a_{\alpha\beta\gamma\delta}=\frac{1}{2\mu}\left[
\frac12(a_{\alpha\delta}a_{\beta\gamma}+
a_{\beta\delta}a_{\alpha\gamma})-\frac{\lambda}{2\mu+3\lambda}a_{\alpha\beta}a_{\gamma\delta}
\right]
\end{equation*}
For symmetric tensors $\sigma^{\alpha\beta}$ and $\gamma_{\alpha\beta}$,
$\sigma^{\alpha\beta}=a^{\alpha\beta\gamma\delta}\gamma_{\gamma\delta}$ if and only if
$\gamma_{\alpha\beta}=a_{\alpha\beta\gamma\delta}\sigma^{\gamma\delta}$.
The elastic tensor is a continuous and positive definite operator in the sense that there is a positive constant $C$ depending
on the shell surface and shell material such that for any covariant tensors $\gamma_{\alpha\beta}$ and $\rho_{\alpha\beta}$,
\begin{equation}\label{elastic-tensor-equiv}
\begin{gathered}
a^{\alpha\beta\gamma\delta}\gamma_{\alpha\beta}\rho_{\gamma\delta}\le C\left(\sum_{\alpha, \beta=1,2}\gamma_{\alpha\beta}^2\right)^{1/2}
\left(\sum_{\alpha, \beta=1,2}\rho_{\alpha\beta}^2\right)^{1/2},\\
\sum_{\alpha, \beta=1,2}\gamma_{\alpha\beta}^2\le C a^{\alpha\beta\gamma\delta}\gamma_{\alpha\beta}\gamma_{\gamma\delta}.
\end{gathered}
\end{equation}
The compliance tensor has the similar property that for any contravariant tensors $\M^{\alpha\beta}$ and $\N^{\alpha\beta}$,
\begin{equation}\label{compliance-tensor-equiv}
\begin{gathered}
a_{\alpha\beta\gamma\delta}\M^{\alpha\beta}\N^{\gamma\delta}\le C\left(\sum_{\alpha, \beta=1,2}{\M^{\alpha\beta}}^2\right)^{1/2}
\left(\sum_{\alpha, \beta=1,2}{\N^{\alpha\beta}}^2\right)^{1/2},\\
\sum_{\alpha, \beta=1,2}{\M^{\alpha\beta}}^2\le C a_{\alpha\beta\gamma\delta}\M^{\alpha\beta}\M^{\gamma\delta}.
\end{gathered}
\end{equation}
The model  \eqref{N-model} has a unique solution in the space $H$ \cite{CiarletIII, BCM}. When $\eps\to 0$, its solution
behaves in very different manners, depending on whether it is bending dominated, membrane/shear dominated, or intermediate.
For bending dominated shell problems, when the resultant loading functions $p^i$, $q^i$, and $r^\alpha$ are independent of $\eps$,
the  model solution converges to a nonzero limit that solves the limiting bending model.
We show below that the scaled membrane stress and scaled shear stress
also converge to  finite limits.

As did in \cite{ABrezzi2, Bramble-Sun2, Suri} for Naghdi shell and in \cite{Koiter-shell} for Koiter shell, 
we split a small portion of the membrane and shear parts and add them to the bending part,
replace $\eps^{-2}-\frac13$  by $\eps^{-2}$,  introduce the scaled membrane stress tensor 
$\M^{\eps\alpha\beta}=
\eps^{-2}a^{\alpha\beta\lambda\gamma}\gamma_{\lambda\gamma}(\!u\e, w\e)$
and the scaled shear stress vector $\xi^{\eps\alpha}=\eps^{-2}\kappa\mu a^{\alpha\beta}\tau_\beta(\!theta\e, \!u\e, w\e)$ 
as new variables, and
write the model in a mixed form. The mixed model seeks $(\!theta\e, \!u\e, w\e)\in H$ and $(\!xi^{\eps}, \M^{\eps})\in V=[L^2]^5$
such that
\begin{multline}\label{N-P-model}
\frac13\int_{\t\Omega}
\left[a^{\alpha\beta\lambda\gamma}\rho_{\lambda\gamma}(\!theta\e, \!u\e, w\e)
\rho_{\alpha\beta}(\!phi, \!v, z)
+
a^{\alpha\beta\lambda\gamma}\gamma_{\lambda\gamma}(\!u\e,w\e)
\gamma_{\alpha\beta}(\!v,z)\right.\\
\left.+\kappa\mu a^{\alpha\beta}\tau_\beta(\!theta\e, \!u\e, w\e)\tau_\alpha(\!phi, \!v, z)\right]
+\int_{\t\Omega}
\left[\M^{\eps\alpha\beta}
\gamma_{\alpha\beta}(\!v,z)
+\xi^{\eps\alpha}\tau_\alpha(\!phi, \!v, z)\right]
\\
\hfill
=\int_{\t\Omega}
(p^{\alpha}v_{\alpha}+
p^3z)
+\int_{\partial^S\t\Omega}r^\alpha\phi_\alpha
+\int_{\partial^F\t\Omega}\left(q^\alpha v_\alpha+q^3z+r^\alpha\phi_\alpha\right)
\ \forall\
(\!phi, \!v,z) \in H,
\\
\int_{\t\Omega}\left[\N^{\alpha\beta}\gamma_{\alpha\beta}(\!u\e, w\e)+\eta^\alpha\tau_\alpha(\!theta\e, \!u\e, w\e)\right]
-
\eps^2\int_{\t\Omega}\left[a_{\alpha\beta\lambda\gamma}\M^{\eps\alpha\beta}\N^{\lambda\gamma}+\frac{1}{\kappa\mu}a_{\alpha\beta}\xi^{\eps\alpha}\eta^\beta\right]=0\\
\hfill\forall\ (\!eta, \N) \in V.
\end{multline}
We repeat that the $\eps^2$ in this mixed formulation is actually $\eps^2/(1-\frac13\eps^2)$ with $\eps$ being the shell half thickness. 
This mixed model  is the basis for the finite element method.
In the next subsection, we present two results in abstract form, which are  applicable to the Naghdi model in the original form \eqref{N-model}
and the mixed form \eqref{N-P-model}, respectively. The latter result also furnishes a framework for analysis of the finite element model.

\subsection{Asymptotic estimates on the shell model}
\label{abstract}
Notations in this sub-section are independent of the rest of the paper.
The proofs of the theorems are given in the appendix. 
The Naghdi shell model \eqref{N-model} can be fitted in the abstract equation \eqref{prob1as} below.
Let $H$,
$U$, and $V$ be Hilbert spaces, $A$ and $B$ be
linear continuous operators from $H$ to
$U$ and $V$, respectively.
We assume
\begin{equation}\label{equiva1as}
\|Au\|_U+\|Bu\|_V\simeq\|u\|_H\ \ \forall \ u\in H.
\end{equation}
For any $\eps>0$ and $f\in H^*$, the dual space of $H$, there is a unique $u\e\in H$, such that
\begin{equation}\label{prob1as}
(Au\e,Av)_U+\eps^{-2}(Bu\e,Bv)_V=\langle f,v\rangle
\quad \forall\ v\in H.
\end{equation}
We let $\ker B\subset H$ be the kernel of the operator $B$, and
let $W\subset V$ be the range of $B$. We define a norm on $W$ by $\|w\|_{W}=\inf_{v\in H, Bv=w}\|v\|_U$ $\forall\ w\in W$, such that
$W$ is isomorphic to $H/\ker B$.
We let $\overline W$ be the closure of $W$ in $V$.
Thus $W$ is a dense subset of $\overline W$, and  $(\overline W)^*$ is dense  in $W^*$.
We need a weaker norm on $W$. For $w\in W$, we define $\|w\|_{\overline{\overline W}}=\|\pi_{\overline W}w\|_{W^*}$. Here $\pi_{\overline W}:\overline W\to(\overline W)^*$
is the inverse of Riesz representation. The relation among these norms is that for any $w\in W$,
$\|w\|_{\overline{\overline W}}\le \|\pi_{\overline W}w\|_{(\overline W)^*}=\|w\|_{\overline W}=\|w\|_V\le \|w\|_W$.
We let $\overline{\overline W}$ be the closure of $W$ in this new norm. This closure is isomorphic to $W^*$. We let $j_{[W^*\to \overline{\overline W}]}$
be the isomorphic mapping from $W^*$ to $\overline{\overline W}$.
We assume that $f|_{\ker B}\ne 0$, such that the limiting problem
\begin{equation}\label{limitas}
(Au^0, Av)_U=\langle f, v\rangle\ \ \forall\ v\in\ker B
\end{equation}
has a nonzero solution $u^0\in\ker B$.
\begin{thm}\label{N-primary-limit}
For the solution of \eqref{prob1as},
we have the asymptotic behavior that $\lim_{\eps\to 0}\|u\e-u^0\|=0$. Furthermore, there is a unique $\M\in \overline{\overline W}$ such that
$\lim_{\eps\to 0}\|\eps^{-2}Bu\e-\M\|_{\overline{\overline W}}=0$.
\end{thm}

In terms of the Naghdi model \eqref{N-model}, the operator $B$ is the membrane and transverse shear strain operator,
and $\ker B$ is the space pure bending deformations. The situation of $\ker B\ne 0$ is that the shell allows pure bendings,
and the condition $f|_{\ker B}\ne 0$
means that the load on the shell indeed activates pure bending.  The convergence described in this theorem
means when $\eps\to 0$, $(\!theta\e, \!u\e, w\e)$ converges to a non-zero limit in $\!H^1\x\!H^1\x H^1$ and the scaled
transverse shear stress vector and  membrane stress tensor $(\!xi\e, \M\e)$ converges to a limit in a space that generally can not be described  in 
the usual sense of space of functions
or distributions. This is a minimum information one should have in order to conceive  a possibility to make the term 
$\sum_{\tau\in\T_h}
h^4_\tau(\sum_{\alpha=1,2}\|\xi^{\eps\alpha}\|^2_{2,\tau}+\sum_{\alpha,\beta=1}^2\|\M^{\eps\alpha\beta}\|^2_{2,\tau})$ in the error estimate \eqref{estimate1} 
small, uniformly with respect to $\eps$, by a limited number of triangles.

The mixed form of the Naghdi shell model \eqref{N-P-model} can be fitted in the abstract problem
\eqref{isomorphism-abs} below. 
Let $H, V$ be Hilbert spaces. Let $a(\cdot,\cdot)$ and $c(\cdot,\cdot)$ be symmetric bilinear forms on $H$ and $V$, and $b(\cdot,\cdot)$
be a bilinear form on $H\x V$. We assume that there is  a constant $C$ such that
\begin{equation}\label{isomorphism-condition}
\begin{gathered}
|a(u, v)|\le C\|u\|_H\|v\|_H, \quad C^{-1}\|u\|_H^2\le a(u, u)\ \forall\ u, v\in H,\\
|c(p, q)|\le C\|p\|_V\|q\|_V, \quad C^{-1}\|p\|_V^2\le c(p, p)\ \forall\ p, q\in V,\\
|b(v, q)|\le C\|v\|_H\|q\|_V\ \forall\ v\in H, q\in V.
\end{gathered}
\end{equation}
For  $f\in H^*$ and $g\in V^*$,
we seek $u\in H$ and $p\in V$ such that
\begin{equation}\label{isomorphism-abs}
\begin{gathered}
a(u, v)+b(v, p)=\langle f, v\rangle\ \ \forall\ v\in H,\\
b(u, q)-\eps^2c(p, q)=\langle g, q\rangle\ \ \forall\ q\in V.
\end{gathered}
\end{equation}
This problem  has a unique solution in the space $H\x V$ \cite{ABrezzi2},
for which we need an accurate estimate. Although this problem has been extensively studied in the literature \cite{Brezzi-book},
we were not able to find what we exactly need. So we include the theorem below.

For $v\in H$, there is a $l(v)\in V^*$ such that $\langle l(v), q\rangle =b(v, q)$ $\forall\ q\in V$.
We let $B(v)=i_Vl(v)\in V$, with $i_V:V^*\to V$ being the Riesz representation operator. Let $W$ be the range of $B$.
We define a weaker (semi) norm on $V$ by
\begin{equation}\label{isomorphism-weak}
|q|_{\overline V}=\sup_{v\in H}\frac{b(v,q)}{\|v\|_H}\ \ \forall\ q\in V.
\end{equation}
If $W$ is dense in $V$,
this is a  weaker norm. Otherwise, it is a semi-norm. 
Whether $W$ is dense in $V$ or not, we have the following
equivalence result.
\begin{thm}\label{isomorphism-thm}
There exist constants $C_\alpha$ that only depend on the constant in \eqref{isomorphism-condition} such that 
\begin{multline}\label{isomorphism-equiv}
\|u\|_H+|p|_{\overline V}+\eps\|p\|_V\le C_1
\sup_{v\in H, q\in V}\frac{a(u, v)+b(v, p)-
b(u, q)+\eps^2c(p, q)}{\|v\|_H+|q|_{\overline V}+\eps\|q\|_V}\\
\le C_2(\|u\|_H+|p|_{\overline V}+\eps\|p\|_V)
\ \ \forall\ u\in H, \ p\in V.
\end{multline}
\end{thm}

\section{The finite element model}
\label{principle}
As mentioned in the introduction, we assume that the coordinate domain $\Omega$ is a polygon.
On $\Omega$, we introduce a triangulation $\T_h$
that is shape regular but not necessarily quasi-uniform.
We use $\E^0_h$ to denote both
the union of interior edges and the set of all interior edges.
The set of edges on the boundary $\partial\Omega$
is denoted by $\E^{\partial}_h$ that is divided as $\E^D_h\cup\E^S_h\cup \E^F_h$, corresponding to clamped, soft-simply supported, and free
portions of the shell boundary. We let
$\E_h=\E^0_h\cup\E^{\partial}_h$. For a $e\in\E^h$, we use $h_e$ to denote its length. We use $\tilde\E_h=\!phi(\E_h)$ to denote the curvilinear
edges on the shell mid-surface $\t\Omega$ in a similar way. 
On $\Omega_h$, for any piecewise vectors $\xi_\alpha$, $\eta_\alpha$, $\theta_\alpha$, $\phi_\alpha$, $u_\alpha$ and $v_\alpha$, 
scalars $w$ and $z$, symmetric tensors
$\M^{\alpha\beta}$ and $\N^{\alpha\beta}$,  we define the following bilinear and linear forms.
\begin{multline}\label{form_ub_a}
\ub a(\!theta, \!u, w; \!phi,  \!v, z)=\\
\frac13\left\{\int_{\t\Omega_h}
\left[a^{\alpha\beta\lambda\gamma}\rho_{\lambda\gamma}(\!theta, \!u, w)
\rho_{\alpha\beta}(\!phi, \!v, z)
+
a^{\alpha\beta\lambda\gamma}\gamma_{\lambda\gamma}(\!u, w)
\gamma_{\alpha\beta}(\!v,z)\right.\right.\\
\left.+\kappa\mu a^{\alpha\beta}\tau_\beta(\!theta, \!u, w)\tau_\alpha(\!phi, \!v, z)\right]
\end{multline}
\begin{multline*}
-\int_{\t\E^0_h}a^{\alpha\beta\lambda\gamma}\lbrac\rho_{\lambda\gamma}(\!phi, \!v, z)\rbrac\lbra\theta_\alpha\rbra_{n_\beta}
-\int_{\t\E^0_h}a^{\alpha\beta\lambda\gamma}\lbrac\rho_{\lambda\gamma}(\!theta, \!u, w)\rbrac\lbra\phi_\alpha\rbra_{n_\beta}\\
-\int_{\t\E^0_h}\kappa\mu a^{\alpha\beta}\lbrac\tau_\beta(\!theta, \!v, z)\rbrac\lbra w\rbra_{n_\alpha}
-\int_{\t\E^0_h}\kappa\mu a^{\alpha\beta}\lbrac\tau_\beta(\!theta, \!u, w)\rbrac\lbra z\rbra_{n_\alpha}
\end{multline*}
\begin{multline*}
+\int_{\t\E^0_h}\left[a^{\alpha\beta\lambda\gamma}\lbrac\rho_{\lambda\gamma}(\!phi, \!v, z)\rbrac b^\delta_\alpha-a^{\delta\beta\alpha\gamma}
\lbrac\gamma_{\alpha\gamma}(\!v, z)\rbrac
\right]
\lbra u_\delta\rbra_{n_\beta}\\
+\int_{\t\E^0_h}\left[a^{\alpha\beta\lambda\gamma}\lbrac\rho_{\lambda\gamma}(\!theta, \!u, w)\rbrac b^\delta_\alpha-a^{\delta\beta\alpha\gamma}
\lbrac\gamma_{\alpha\gamma}(\!u, w)\rbrac
\right]
\lbra v_\delta\rbra_{n_\beta}
\end{multline*}
\begin{multline*}
+\int_{\t\E^{D\cup S}_h}\left[a^{\alpha\beta\lambda\gamma}\rho_{\lambda\gamma}(\!phi, \!v, z)b^\delta_\alpha-a^{\delta\beta\alpha\gamma}
\gamma_{\alpha\gamma}(\!v, z)
\right]
u_\delta{n_\beta}\\
+\int_{\t\E^{D\cup S}_h}\left[a^{\alpha\beta\lambda\gamma}\rho_{\lambda\gamma}(\!theta, \!u, w)b^\delta_\alpha-a^{\delta\beta\alpha\gamma}
\gamma_{\alpha\gamma}(\!u, w)
\right]
v_\delta{n_\beta}\\
-\int_{\t\E^{D\cup S}_h}\kappa\mu a^{\alpha\beta}\tau_\beta(\!phi, \!v, z)w{n_\alpha}
-\int_{\t\E^{D\cup S}_h}\kappa\mu a^{\alpha\beta}\tau_\beta(\!theta, \!u, w)z{n_\alpha}\\
\left.
-\int_{\t\E^D_h}a^{\alpha\beta\lambda\gamma}\rho_{\lambda\gamma}(\!phi, \!v, z)\theta_\alpha{n_\beta}
-\int_{\t\E^D_h}a^{\alpha\beta\lambda\gamma}\rho_{\lambda\gamma}(\!theta, \!u, w)\phi_\alpha{n_\beta}\right\}.
\end{multline*}
An inter-element edge $\t e\in\t\E^0_h$ is shared by elements $\t\tau_1$ and $\t\tau_2$. Piecewise
functions may have different values on the two elements, and thus discontinuous on $\t e$. The notation
$\lbrac\rho_{\sigma\tau}(\!phi, \!v, z)\rbrac$ represents the average of values of  $\rho_{\sigma\tau}(\!phi, \!v, z)$
from the sides of $\t\tau_1$ and $\t\tau_2$.
Let $\!n_\delta=n^\alpha_\delta\!a_\alpha=n_{\delta\alpha}\!a^\alpha$ be the unit outward normal to $\t e$ viewed as boundary of $\t\tau_\delta$.
We have $\!n_1+\!n_2=0$, $n^\alpha_1+n^\alpha_2=0$, and  $n_{1\alpha}+n_{2\alpha}=0$.
We use $\lbra u_{\alpha}\rbra_{n_{\lambda}}=(u_{\alpha})|_{\t\tau_1} n_{1\lambda}+(u_{\alpha})|_{\t\tau_2} n_{2\lambda}$ to denote the jump of $u_\alpha$ over the edge
$\t e$ with respect to $n_\lambda$, 
and use $\lbra w\rbra_{n_{\lambda}}=w|_{\t\tau_1} n_{1\lambda}+w|_{\t\tau_2} n_{2\lambda}$ to denote the jump of $w$ over the edge
$\t e$ with respect to $n_\lambda$, 
etc. 
On the boundary $\t\E^{D\cup S}_h$,  $\!n=n^\alpha\!a_\alpha=n_\alpha\!a^\alpha$ is the unit outward in surface normal to $\partial\t\Omega$.

We add  some additional penalty terms on the 
inter-element discontinuity and on the clamped and simply supported portions of the
boundary, and define a symmetric
bilinear form $a(\!theta, \!u, w;\!phi, \!v, z)$ by
\begin{multline}\label{form_a}
a(\!theta, \!u, w; \!phi, \!v, z)=\ub a(\!theta, \!u, w; \!phi, \!v, z)\\
+\C\sum_{e\in\E^0_h}
h^{-1}_e\int_{e}\left[\sum_{\alpha=1,2}\left(\lbra u_{\alpha}\rbra \lbra v_{\alpha}\rbra+\lbra \theta_{\alpha}\rbra \lbra \phi_{\alpha}\rbra\right)
+\lbra w\rbra \lbra z\rbra\right]
\\
+\C\sum_{e\in\E^{D\cup S}_h}
h^{-1}_e\int_{e}\left(\sum_{\alpha=1,2}\lbra u_{\alpha}\rbra \lbra v_{\alpha}\rbra
+\lbra w\rbra \lbra z\rbra\right)
+
\C\sum_{e\in\E^{D}_h}
h^{-1}_e\int_{e}\sum_{\alpha=1,2}\lbra \theta_{\alpha}\rbra \lbra \phi_{\alpha}\rbra.
\end{multline}
The jump $\lbra u_\alpha\rbra$ is the absolute value of the difference 
in the values of $u_\alpha$ from the two sides of $e$. The jumps $\lbra\theta_\alpha\rbra$ and  $\lbra w\rbra$ are defined in the same way. 
We also define the bilinear forms 
\begin{multline}\label{form_b}
b(\M,\!xi; \!phi, \!v,z)=
\int_{\t\Omega_h}\left[\M^{\alpha\beta}\gamma_{\alpha\beta}(\!v, z)+\xi^\alpha\tau_\alpha(\!phi, \!v, z)\right]\\
-\int_{\t\E^0_h}\left(\lbrac\M^{\alpha\beta}\rbrac\lbra v_{\alpha}\rbra_{n_{\beta}}+\lbrac\xi^\alpha\rbrac\lbra z\rbra_{n_\alpha}\right)
 -\int_{\t\E^{D\cup S}_h}\left(
\M^{\alpha\beta}{n_{\beta}}v_{\alpha}+\xi^\alpha n_\alpha z\right),
\end{multline}
\begin{multline}\label{form_c}
c(\M,\!xi; \N,\!eta)=\int_{\t\Omega_h}\left(a_{\alpha\beta\gamma\delta}\M^{\gamma\delta}\N^{\alpha\beta}
+\frac{1}{\kappa\mu}a_{\alpha\beta}\xi^\alpha\eta^\beta\right).\hfill
\end{multline}
We define a linear form
\begin{multline}\label{form_f}
\langle \!f; \!phi, \!v, z\rangle=\int_{\t\Omega_h}
(p^{\alpha}v_{\alpha}+
p^3z)
+\int_{\t\E^S_h}r^\alpha\phi_\alpha
+\int_{\t\E^F_h}\left(q^\alpha v_\alpha+q^3z+r^\alpha\phi_\alpha\right).\hfill
\end{multline}
All these forms are well defined for piecewise functions that could be independently defined on each element of the triangulation $\T_h$.

The finite element model is defined on a space of piecewise polynomials. We use continuous piecewise linear polynomials for the components
$\M^{\alpha\beta}$
of the scaled membrane stress tensor $\M$ and for the components $\xi^\alpha$ of the scaled transverse shear stress vector $\!xi$.
Therefore, in the right hand side of \eqref{form_b}, the averages  $\lbrac\M^{\alpha\beta}\rbrac$ and $\lbrac\xi^\alpha\rbrac$
are replaced by the function values $\M^{\alpha\beta}$ and $\xi^\alpha$, respectively.
We use discontinuous
piecewise quadratic polynomials for components $u_\alpha$ of the tangential displacement vector $\!u$, components $\theta_\alpha$ of the rotation 
vector $\!theta$, and the scalar $w$ of normal displacement. 
The finite element space for the displacement components $u_\alpha$ and $w$ 
needs to be enriched  on elements that have one or two edges
on the free boundary $\E^F_h$. There is no need to enrich the finite element space for rotations.
On an element $\tau$ (with edges $e_i$), we let $P^k(\tau)$ be the space
of polynomials of degree $k$.  If $\tau$ has one edge ($e_1$) on the free boundary,
we define two cubic polynomials $p^3_\alpha$ by
\begin{equation*}
p^3_1=\lambda_1p^2_1+1,\quad p^3_2=\lambda_1p^2_2+\lambda_2.
\end{equation*}
Here $\lambda_i$ are the barycentric coordinates with respect to the vertex opposite to the edge $e_i$, and $p^2_\alpha\in P^2(\tau)$ are
defined by
\begin{equation}\label{P3*}
\int_\tau(\lambda_1p^2_1+1)q\sqrt a=0\ \forall\ q\in P^2(\tau), \quad
\int_\tau(\lambda_1p^2_2+\lambda_2)q\sqrt a=0\ \forall\ q\in P^2(\tau).
\end{equation}
Note that functions in $\text{span}(p^3_\alpha)$ are orthogonal to $P^2(\tau)$ with respect to the inner product of 
$L^2(\tau)$ weighted by $\sqrt a$,
and they are linear on $e_1$, on which they can be determined by their zero and first moments weighted by $\sqrt a$.
We then define the $P^3_*(\tau)=P^2(\tau)\oplus\text{span}(p^3_\alpha)$.
If $\tau$ has two edges ($e_\alpha$) on the free boundary, we take the full $P^3(\tau)$.
A function in $P^3(\tau)$ is uniquely determined by its projection into  $P^2(\tau)$ with respect to $L^2(\tau)$ weighted by $\sqrt a$
and its zero and first moments on $e_1$ and $e_2$ weighted by $\sqrt a$. 
These will be explicitly given by the formulas \eqref{thetaI} -- \eqref{uwI-edge2} below.
Let $P^u(\tau)= P^2(\tau)$, $P^3_*(\tau)$, or $P^3(\tau)$, depending on
whether  $\tau$ has no edge, one edge, or two edges on the free boundary $\E^F_h$.
The finite element space is defined by
\begin{equation}\label{FE-space}
\begin{gathered}
\H_h=\{(\!phi, \!v, z);\text{ on each }\tau\in\T_h,\  \phi_\alpha\in P^2(\tau), v_\beta\in P^u(\tau), z\in P^u(\tau)\},\\
\V_h=\{(\N, \!eta);\  \N^{\alpha\beta},\eta^\gamma\in H^1,  \text{ on each }\tau\in\T_h,\   \N^{\alpha\beta}, \eta^\gamma\in P^1(\tau)\}.
\end{gathered}
\end{equation}
The finite element model seeks $(\!theta, \!u, w)\in \H_h$ and $(\M, \!xi)\in \V_h$ such that
\begin{equation}\label{N-fem}
\begin{gathered}
a(\!theta, \!u, w;\  \!phi, \!v, z)+b(\M, \!xi;\  \!phi, \!v, z)=\langle\!f;\!phi, \!v, z\rangle\ \ \forall\ (\!phi, \!v, z)\in \H_h, \\
b(\N, \!eta;\  \!theta, \!u, w)-\eps^2c(\M, \!xi;\  \N,\!eta)=0\ \ \forall\ (\N, \!eta)\in \V_h.
\end{gathered}
\end{equation}
This equation is in the form of \eqref{isomorphism-abs}. We shall define the norms in $\H_h$ and $\V_h$ later,
in which we prove that the finite element model \eqref{N-fem} is well posed
if the penalty constant $\C$ in \eqref{form_a} is sufficiently large, by verifying the conditions \eqref{isomorphism-condition}.
This penalty constant $\C$ could be dependent on the shell geometry and the shape regularity $\K$ of
the triangulation $\T_h$. It is, otherwise, independent of the triangulation $\T_h$ and the shell thickness.

The solution of the Naghdi  model \eqref{N-P-model} satisfies the equation \eqref{N-fem} when the test functions $\!phi$, $\!v$, $z$, $\N$, and $\!eta$
are arbitrary piecewise smooth functions, not necessarily polynomials. This says that the finite element model is consistent
with the shell model. The consistency is verified in the appendix.

Note that there is no boundary condition enforced on functions in the spaces $\H_h$ and $\V_h$.
The displacement boundary condition is enforced by boundary penalty in a consistent manner, which is Nitsche's method
\cite{A-DG}.
For displacement components on an element with one edge on $\E^F_h$, one may simply replace $P^3_*(\tau)$  by the richer $P^3(\tau)$.
This would slightly increase the complexity, but not affect the stability or accuracy of the finite element method.
The finite element space for the rotation could be taken as continuous piecewise quadratic polynomials. The error estimate will not be changed.
In this case the  terms in the bilinear forms \eqref{form_ub_a} and \eqref{form_a} that have a factor of the form $\lbra\theta_\alpha\rbra$ or 
$\lbra\phi_\alpha\rbra$ would be replaced  by zero. On $\E^D_h$, the zero boundary condition for the rotation variable also needs 
to be explicitly enforced.

\section{A discrete Korn's inequality for Naghdi shell}
\label{KornOnShell}
To prove the continuity and coerciveness of the bilinear form \eqref{form_a}
for finite element functions in a suitable space, we need to have a Korn type inequality
that bounds a discrete $H^1$ norm of the displacement and rotation variables by 
the $L^2$ norms of the bending, membrane, and transverse shear strains.

Let $H^1_h$ be the space of piecewise $H^1$ functions in which a function is independently defined on each element $\tau$, and
$u|_\tau\in H^1(\tau)$ for $\tau\in\T_h$. We define a norm in this space by 
\begin{equation}\label{N-h-norm-u}
\|u\|_{H^1_h}=\left(\sum_{\tau\in\T_h}\|u\|^2_{1,\tau}
+\sum_{e\in \E^0_h}
h^{-1}_e\int_{e}\lbra u\rbra^2\right)^{1/2}.
\end{equation}
For $\theta_\alpha$, $u_\alpha$ and $w$ in $H^1_h$, we define a norm
\begin{equation}\label{h-norm}
\|(\!theta, \!u, w)\|_{\!H^1_h\x \!H^1_h\x H^1_h}=\left[\sum_{\alpha=1,2}\left(\|\theta_\alpha\|^2_{H^1_h}+\|u_\alpha\|^2_{H^1_h}\right)+\|w\|^2_{H^1_h}\right]^{1/2}.
\end{equation}
Let $f(\!theta, \!u, w)$ be a semi-norm that is continuous with respect to this norm such that there is a $C$ only dependent on the shape regularity 
$\K$ of $\T_h$ and shell mid surface such that
\begin{equation}\label{f-continuous}
|f(\!theta, \!u, w)|\le C\|(\!theta, \!u, w)\|_{\!H^1_h\x \!H^1_h\x H^1_h}\ \forall\ (\!theta, \!u, w)\in \!H^1_h\x \!H^1_h\x H^1_h.
\end{equation}
We also assume that $f$ satisfies the condition that
if $(\!theta, \!u, w)\in  \!H^1\x \!H^1\x H^1$ defines a rigid body motion and $f(\!theta, \!u, w)=0$ then $\!theta=0$, $\!u=0$, and $w=0$.
The space of rigid body motion is a $6$-dimensional space, denoted by $RBM$. The functions $(\!theta, \!u, w)\in RBM$ if and only if
$u_\alpha\!a^\alpha+w\!a^3$ is a rigid body motion of the shell midsurface, and $\tau_\alpha(\!theta, \!u, w)=0$. This is also equivalent to that 
$\rho_{\alpha\beta}(\!theta, \!u, w)=0$, $\gamma_{\alpha\beta}(\!u, w)=0$, and $\tau_\alpha(\!theta, \!u, w)=0$ \cite{BCM}.
Therefore, 
we have another norm on the space $\!H^1_h\x\!H^1_h\x H^1_h$ defined by
\begin{multline}\label{triple-norm}
\ll(\!theta, \!u, w)\ll^2_h=
\sum_{\alpha,\beta=1,2}\left(\|\rho_{\alpha\beta}(\!theta, \!u, w)\|^2_{0,\Omega_h}+\|\gamma_{\alpha\beta}(\!u, w)\|^2_{0,\Omega_h}\right)+
\sum_{\alpha=1,2}\|\tau_\alpha(\!theta, \!u, w)\|^2_{0,\Omega_h}
\\
+\sum_{e\in \E^0_h}h^{-1}_e\left[
\sum_{\alpha=1,2}\int_{e}\left(\lbra \theta_{\alpha}\rbra^2+\lbra u_{\alpha}\rbra^2\right)
+\int_{e}\lbra w\rbra^2\right]
+f^2(\!theta,\!u, w).
\end{multline}
We have the following Korn's inequality for  piecewise functions.
\begin{thm}\label{Korn-thm}
There exists a constant $C$ that could be dependent on the shell mid-surface
and shape regularity $\K$ of the triangulation $\T_h$, but otherwise independent of the triangulation,
such that
\begin{equation}\label{Korn-inequality}
\|(\!theta, \!u, w)\|_{\!H^1_h\x\!H^1_h\x H^1_h}\le C\ll(\!theta, \!u, w)\ll_h\ \forall\ (\!theta, \!u, w)\in \!H^1_h\x\!H^1_h\x H^1_h.
\end{equation}
\end{thm}
In view of the definitions \eqref{N-bending}, \eqref{N-metric}, and \eqref{N-shear}, 
this theorem implies that 
\begin{equation*}
\|(\!theta, \!u, w)\|_{\!H^1_h\x\!H^1_h\x H^1_h}\simeq \ll(\!theta, \!u, w)\ll_h.
\end{equation*}
To prove the inequality, we need a discrete Korn's inequality for piecewise functions in $H^1_h$, see (1.21) of \cite{Brenner-Korn}.
It says that
that there is a constant $C$ that might be dependent on the domain $\Omega$ and the shape regularity $\K$ of the triangulation $\T_h$, but
otherwise independent of $\T_h$ such that
\begin{equation}\label{Korn-Brenner}
\sum_{\alpha=1,2}\|u_\alpha\|^2_{H^1_h}\le C\left[\sum_{\alpha=1,2}\|u_\alpha\|^2_{0,\Omega_h}
+\sum_{\alpha, \beta=1,2}\|e_{\alpha\beta}(\!u)\|^2_{0,\Omega_h}+\sum_{e\in\E^0_h}
h^{-1}_e\int_{e}\sum_{\alpha=1,2}\lbra u_{\alpha}\rbra^2\right].
\end{equation}
Here $e_{\alpha\beta}(\!u)=(\partial_\beta u_\alpha+\partial_\alpha u_\beta)/2$ is the symmetric part of the gradient of $\!u$.
It follows from this  inequality and the definitions \eqref{N-bending}, \eqref{N-metric}, and \eqref{N-shear}
of $\rho_{\alpha\beta}$, $\gamma_{\alpha\beta}$, and $\tau_\alpha$
that there is a constant $C$ that only depends on the shell midsurface and shape regularity of $\T_h$ such that
\begin{multline}\label{Korn-thm-proof0}
\|(\!theta, \!u, w)\|_{\!H^1_h\x\!H^1_h\x H^1_h}^2\le C\left\{
\sum_{\alpha,\beta=1,2}\left(\|\rho_{\alpha\beta}(\!theta, \!u, w)\|^2_{0,\Omega_h}+\|\gamma_{\alpha\beta}(\!u, w)\|^2_{0,\Omega_h}\right)\right. \\+
\sum_{\alpha=1,2}\|\tau_\alpha(\!theta, \!u, w)\|^2_{0,\Omega_h}
+\sum_{e\in \E^0_h}h^{-1}_e\left[
\sum_{\alpha=1,2}\int_{e}\left(\lbra \theta_{\alpha}\rbra^2+\lbra u_{\alpha}\rbra^2\right)
+\int_{e}\lbra w\rbra^2\right]\\
\left.+\sum_{\alpha=1,2}\left(\|\theta_\alpha\|^2_{0,\Omega_h}+\|u_\alpha\|^2_{0,\Omega_h}\right)+\|w\|^2_{0,\Omega_h}\right\}
\ \ \forall\ (\!theta, \!u, w)\in \!H^1_h\x\!H^1_h\x H^1_h.
\end{multline}

We also need a trace theorem and a compact embedding result for functions in $H^1_h$.
\begin{lem}
Let $\tau$ be a triangle, and $e$ one of its edges. Then there is a
constant $C$ depending on the shape regularity of $\tau$ such that
\begin{equation}\label{trace}
\int_eu^2\le C\left[h_e^{-1}\int_{\tau}u^2+\sum_{\alpha=1,2}h_e\int_{\tau}|\partial_\alpha u|^2\right]
\ \ \forall\ u\in H^1(\tau).
\end{equation}
\end{lem}
This can be found in \cite{A-DG}.
For piecewise functions in $H^1_h$, we have the following trace theorem.
\begin{lem}\label{tracetheorem}
There
exists a constant $C$ depending on $\Omega$ and the shape regularity
$\K$ of $\T_h$, but otherwise independent of the triangulation such that
\begin{equation}\label{Omega-trace}
\|u\|_{L^2(\partial\Omega)}\le C \|u\|_{H^1_h}\ \ \forall\ u\in H^1_h.
\end{equation}
\end{lem}
\begin{proof}
Let $\!phi$ be a piecewise smooth
vector field on $\Omega$ whose normal component is continuous across any straight line segment, and
such that $\!phi\cdot\bar\!n=1$ on $\partial\Omega$. (The piecewise smoothness of $\!phi$ is not
associated with the triangulation $\T_h$. A construction of such vector field is given below.)
On each element
$\tau\in\T_h$, we have
\begin{equation*}
\int_{\partial\tau}u^2\!phi\cdot\bar\!n=\int_\tau \div(u^2\!phi)=
\int_\tau(2u\nabla u\cdot\!phi+u^2\div\!phi).
\end{equation*}
Summing up over all elements of $\T_h$, we get
\begin{equation*}
\int_{\partial\Omega}u^2=-\sum_{e\in\E^0_h}\int_e\lbra u^2\!phi\rbra
+\int_{\Omega_h}(2u\nabla u\cdot\!phi+u^2\div\!phi).
\end{equation*}
If $e$ is the border between the elements
$\tau_1$ and
$\tau_2$ with outward normals $\bar\!n_1$ and $\bar\!n_2$, then
$\lbra u^2\!phi\rbra=u^2_1\!phi_1\cdot\bar\!n_1+u^2_2\!phi_2\cdot\bar\!n_2$,
where $u_1$ and $u_2$ are restrictions of $u$
on $\tau_1$ and $\tau_2$, respectively.  It is noted that although $\!phi$ may be discontinuous
across $e$, it normal component is continuous, i.e., $\!phi_1\cdot\bar\!n_1+\!phi_2\cdot\bar\!n_2=0$.
On the edge $e$, we have
$|\lbra u^2\!phi\rbra|\le|\lbra u^2\rbra|\|\!phi\|_{0,\infty,\Omega}$.
Here, $|\lbra u^2\rbra|=|u_1^2-u_2^2|$. It is noted that $|\lbra u^2\rbra|=2|\lbra u\rbra\lbrac u\rbrac|$,
with $\lbrac u\rbrac=(u_1+u_2)/2$ being the average. We have
\begin{multline}\label{jump-on-e}
\int_e|\lbra u^2\!phi\rbra|\le
2|\!phi|_{0,\infty,\Omega}
\left[|e|^{-1}\int_e\lbra u\rbra^2\right]^{1/2}\left[|e|\int_e\lbrac u\rbrac^2\right]^{1/2}\\
\le C
|\!phi|_{0,\infty,\Omega}\left[|e|^{-1}\int_e\lbra u\rbra^2\right]^{1/2}
\left[\sum_{\delta=1,2}\left(\int_{\tau_\delta}u^2+|e|^2\int_{\tau_\delta}|\nabla u|^2\right)\right]^{1/2}.
\end{multline}
Here, $C$ only depends on the shape regularity of $\tau_1$ and $\tau_2$,
and we used the trace
estimate \eqref{trace}.
It then follows from the Cauchy--Schwarz inequality that
\begin{equation*}
\|u\|^2_{L^2(\partial\Omega)}\le C(|\!phi|_{0,\infty,\Omega}+|\div\!phi|_{0,\infty,\Omega})
\left[\|u\|^2_{L^2(\Omega)}+\int_{\Omega_h}|\nabla u|^2
+\sum_{e\in\E^0_h}\frac{1}{|e|}\int_e\lbra u\rbra^2\right].
\end{equation*}
Here the constant $C$ only depends on the shape regularity of $\T_h$.
The dependence on $\Omega$ of the $C$ in \eqref{Omega-trace}
is hidden in the $\!phi$ in the above inequality.
\end{proof}
%
%
We describe a construction of the vector field $\!phi$ used in the proof.
On the $xy$-plane, we consider a triangle $OAB$ with the origin being its
vertex $O$. Let the distance from $O$ to the side $AB$ be $H$. Then the field
$\!psi(x, y)=\langle x, y\rangle/H$ has the property that $\!psi\cdot\bar\!n=1$ on $AB$
and $\!psi\cdot\bar\!n=0$ on $OA$ and $OB$. Also $|\!psi|_{0,\infty}=\max\{|OA|, |OB|\}/H$
and $\div\!psi=2/H$.
For each straight segment of $\partial\Omega$, we define a triangle
with the straight segment being a side
whose opposite vertex is in $\Omega$, then we define a vector field on this triangle
as on the triangle $OAB$ with $AB$ being the straight side.
We need to assure that all such triangles do not overlap. We then piece together
all these vector fields and fill up the remaining part of the domain by a zero
vector field. This defines the desired vector field used in the proof.

The following compact embedding theorem can be derived from a result in \cite{DiPietro},
for which a direct proof can be found in \cite{compact}.
\begin{lem}\label{uniformcompactembedding}
Let $\T_{h_i}$ be a (infinite) class of shape regular but not necessarily quasi-uniform
triangulations of the polygonal domain $\Omega$, with a shape regularity constant $\K$.
For each $i$, let $H^1_{h_i}$ be the space of piecewise $H^1$ functions, subordinated to the
triangulation $T_{h_i}$,
equipped with the norm \eqref{N-h-norm-u}.
Let $\{u_i\}$ be a bounded sequence such that $u_i\in H^1_{h_i}$ for each $i$. I.e.,
there is a constant $C$, such that $\|u_i\|_{H^1_{h_i}}\le C$ for all $i$.
Then, the sequence $\{u_i\}$ has a convergent subsequence in $L^2$.
\end{lem}
\begin{proof}[Proof of Theorem~\ref{Korn-thm}]
From \eqref{Korn-thm-proof0} and the definition \eqref{triple-norm}, it is trivial to see that
that there is a constant $C$ such that
\begin{multline}\label{Korn-thm-proof1}
\|(\!theta, \!u, w)\|_{\!H^1_h\x\!H^1_h\x H^1_h}^2\le C\left[\ll(\!theta,\!u, w)\ll_h^2+\sum_{\alpha=1,2}\left(\|\theta_\alpha\|^2_{0,\Omega_h}+\|u_\alpha\|^2_{0,\Omega_h}\right)
+\|w\|^2_{0,\Omega_h}\right]\\
\forall\ (\!theta, \!u, w)\in \!H^1_h\x\!H^1_h\x H^1_h.
\end{multline}
On a fixed triangulation $\T_h$, it then follows from the Rellich--Kondrachov
compact embedding theorem and Peetre's lemma (Theorem 2.1, page 18 in \cite{Raviart}) that
there is a constant $C_{\T_h}$ such that
\begin{equation*}
\|(\!theta, \!u, w)\|_{\!H^1_h\x\!H^1_h\x H^1_h}\le C_{\T_h}\ll(\!theta,\!u, w)\ll_h\ \ \forall\ (\!theta, \!u, w)\in \!H^1_h\x\!H^1_h\x H^1_h.
\end{equation*}
We need to show that for a class  of shape regular triangulations, such $C_{\T_h}$ has a upper bound that only
depends on the shape regularity $\K$ of the whole class. Otherwise, there exists a sequence
of triangulations $\{\T_{h_n}\}$ and an associated sequence of functions $(\!theta^n, \!u^n, w^n)$ in
$\!H^1_{h_n}\x\!H^1_{h_n}\x H^1_{h_n}$
such that
\begin{equation*}
\|(\!theta^n, \!u^n, w^n)\|_{\!H^1_{h_n}\x\!H^1_{h_n}\x H^1_{h_n}}=1 \text{ and  }\ll(\!theta^n,\!u^n, w^n)\ll_{h_n}\le 1/n.
\end{equation*}
It follows from Lemma~\ref{uniformcompactembedding}
that there is a subsequence,
still denoted by $(\!theta^n, \!u^n, w^n)$, that converges to $(\!theta^0, \!u^0, w^0)$ in $\!L^2\x\!L^2\x L^2$.
We show that this limit defines a rigid body motion and it is zero, which will lead to a contradiction.

First, we show that $w^0$ is actually in $H^1$ and we have that
$\partial_\alpha w^0+\theta^0_\alpha+b^\beta_\alpha u^0_\beta=0$.
For a compactly supported smooth function $\phi$, we have
\begin{equation*}
\int_{\Omega}w^0\partial_\alpha\phi=\lim_{n\to\infty}\int_{\Omega}w^n\partial_\alpha\phi.
\end{equation*}
For each $n$,
\begin{equation*}
\int_{\Omega}w^n\partial_\alpha\phi=
-\int_{\Omega_{h_n}}(\partial_\alpha w^n+\theta^n_\alpha+b^\beta_\alpha u^n_\beta)\phi
+\int_{\Omega_{h_n}}(\theta^n_\alpha+b^\beta_\alpha u^n_\beta)\phi+
\sum_{e\in\E^0_{h_n}}\int_e\lbra w^n\rbra_{\bar n_\alpha}\phi.
\end{equation*}
Here, $\langle \bar n_1, \bar n_2\rangle$ is the unit normal to the edge $e$ and
$\lbra w^n\rbra_{\bar n_\alpha}$ is the jump of $w^n$ over $e$ with respect to $\bar n_\alpha$.
The sum of the first and third terms in the right hand side is bounded by
\begin{equation*}
C\|\tau_\alpha(\!theta^n,\!u^n,w^n)\|_{0,\Omega_{h_n}}\|\phi\|_{0,\Omega}+
C\left[\sum_{e\in\E^0_{h_n}}h^{-1}_e\int_e\lbra w^n\rbra^2\right]^{1/2}\left[|\phi|^2_{0,\Omega}+
\sum_{\tau\in\T_{h_n}}h^2_\tau|\phi|^2_{1,\tau}\right]^{1/2},
\end{equation*}
where $C$ depends on $\K$ only. Since $\ll(\!theta^n,\!u^n, w^n)\ll_h\to 0$, this upper bound tends to zero as $n\to\infty$.
The second term converges to 
$\int_{\Omega}(\theta^0_\alpha+b^\beta_\alpha u^0_\beta)\phi$.
Thus we have
\begin{equation*}
\int_{\Omega}w^0\partial_\alpha\phi=\int_{\Omega}(\theta^0_\alpha+b^\beta_\alpha u^0_\beta)\phi.
\end{equation*}
This shows that $w^0\in H^1$ and 
\begin{equation}\label{shear=0}
\partial_\alpha w^0+\theta^0_\alpha+b^\beta_\alpha u^0_\beta=0.
\end{equation}

Next we show  that $u^0_\alpha\!a^\alpha+w^0\!a^3$ is a rigid body motion of the shell mid-surface.
Let $Q^i$ be a smooth surface force field that annihilates $RBM$, i.e.,
\begin{equation*}
\int_{\t\Omega}(Q^{\alpha}u_{\alpha}+Q^3w)=0 \ \forall\ (\!theta, \!u, w)\in RBM.
\end{equation*}
Such a force field is an admissible loading on the shell with totally free boundary.
We let $m^{\alpha\beta}$, $n^{\alpha\beta}$, and $k^\alpha$  be a stress resultant, stress couple, and transverse shear resultant  equilibrating $Q^i$ such that
\begin{multline*}
\int_{\t\Omega}\left[n^{\alpha\beta}\rho_{\alpha\beta}(\!phi, \!v, z)+m^{\alpha\beta}\gamma_{\alpha\beta}(\!v, z)+k^\alpha\tau_\alpha(\!phi, \!v, z)\right]\\
=
\int_{\t\Omega}(Q^\alpha v_{\alpha}+Q^3z)\ \ \forall\ (\!phi, \!v, z)\in \!H^1\x\!H^1\x H^1.
\end{multline*}
One can choose $m^{\alpha\beta}$, $n^{\alpha\beta}$, and $k^\alpha$ in the following manner.
We consider a Naghdi shell with the mid surface $\t\Omega$, but of unit thickness $\eps=1$, loaded by
$Q^i$, and free on its entire lateral boundary. The shell model has a unique solution $(\!theta_1, \!u_1, w_1)$
in the quotient
space $(\!H^1\x\!H^1\x H^1)/RBM$. We then take
\begin{equation*}
n^{\alpha\beta}=\frac13a^{\alpha\beta\lambda\gamma}\rho_{\lambda\gamma}(\!theta_1, \!u_1, w_1),\quad
m^{\alpha\beta}=a^{\alpha\beta\lambda\gamma}\gamma_{\lambda\gamma}(\!u_1, w_1),\quad
k^\alpha=\kappa\mu a^{\alpha\beta}\tau_\beta(\!theta_1, \!u_1, w_1).
\end{equation*}
Since this $(\!theta_1, \!u_1, w_1)$ is the solution of a rather regular elliptic equation, under our assumption on the regularity of the shell model, 
we have $(\!theta_1, \!u_1, w_1)\in \!H^2\x\!H^2\x H^2$.
Using the Green's theorem on surfaces \eqref{Green}, in view of the definitions \eqref{N-bending}, \eqref{N-metric}, and \eqref{N-shear},
for an element $\tau\in\T_{h_n}$, we have
\begin{multline}\label{shell-equilibrium}
\int_{\t\tau}\left[n^{\alpha\beta}\rho_{\alpha\beta}(\!phi, \!v, z)+m^{\alpha\beta}\gamma_{\alpha\beta}(\!v, z)+k^\alpha\tau_\alpha(\!phi, \!v, z)\right]
=
\\\int_{\t\tau}\left\{\left(-n^{\alpha\beta}|_{\beta}+k^\alpha\right)\phi_{\alpha}
+
\left[-m^{\alpha\beta}|_{\beta}+\left(n^{\lambda\gamma}b^{\alpha}_{\lambda}\right)|_\gamma+
k^\beta b^\alpha_\beta
\right]v_{\alpha}+
\left(-k^\alpha|_\alpha+n^{\alpha\beta}c_{\alpha\beta}-m^{\alpha\beta}b_{\alpha\beta}
\right)z\right\}\\
+\int_{\partial\t\tau}\left[n^{\alpha\beta}n_{\beta}\phi_{\alpha}+
\left(m^{\alpha\beta}n_{\beta}-n^{\lambda\gamma}b^{\alpha}_{\lambda}n_{\gamma}
\right)v_{\alpha}+
k^\alpha n_{\alpha}z\right].
\end{multline}
This identity is also valid when $\tau$ is replaced by $\Omega$, and from that we get the following equilibrium equations and boundary conditions.
\begin{equation}\label{m-n-equilibrium}
\begin{gathered}
-n^{\alpha\beta}|_{\beta}+k^\alpha=0\ \text{ in }\Omega,\\
-m^{\alpha\beta}|_{\beta}+\left(n^{\lambda\gamma}b^{\alpha}_{\lambda}\right)|_\gamma+
k^\beta b^\alpha_\beta=Q^\alpha\ \text{ in }\Omega,\\
-k^\alpha|_\alpha+n^{\alpha\beta}c_{\alpha\beta}-m^{\alpha\beta}b_{\alpha\beta}=Q^3\ \text{ in }\Omega,\\
n^{\alpha\beta}n_{\beta}=0,\quad
m^{\alpha\beta}n_{\beta}-n^{\lambda\gamma}b^{\alpha}_{\lambda}n_{\gamma}=0,\quad
k^\alpha n_{\alpha}=0\ \text{ on }\partial\Omega.
\end{gathered}
\end{equation}
Since $(\!u^n, w^n)\to (\!u^0, w^0)$ in $L^2$, we have
\begin{equation*}
\int_{\t\Omega}(Q^{\alpha}u^0_{\alpha}+Q^3w^0)
=\lim_{n\to\infty}\int_{\t\Omega_{h_n}}(Q^\alpha u^n_{\alpha}+Q^3w^n).
\end{equation*}
For a given $n$, by using \eqref{shell-equilibrium} on each element of $\t\tau\subset\t\Omega_{h_n}$, and summing up, we get
\begin{multline*}
\int_{\t\Omega_{h_n}}(Q^\alpha u^n_{\alpha}+Q^3w^n)
=\int_{\t\Omega_{h_n}}\left\{
\theta^n_\alpha\left(-n^{\alpha\beta}|_{\beta}+k^\alpha\right)+
u^n_{\alpha}
\left[
-m^{\alpha\beta}|_{\beta}+\left(n^{\lambda\gamma}b^{\alpha}_{\lambda}\right)|_\gamma+
k^\beta b^\alpha_\beta
\right]\right\}\\
\hfill
+\int_{\t\Omega_{h_n}}w^n
\left(
-k^\alpha|_\alpha+n^{\alpha\beta}c_{\alpha\beta}-m^{\alpha\beta}b_{\alpha\beta}
\right)\\
=\sum_{\tau\in\T_{h_n}}
\int_{\t\tau}\left[n^{\alpha\beta}\rho_{\alpha\beta}(\!theta^n, \!u^n, w^n)+m^{\alpha\beta}\gamma_{\alpha\beta}(\!u^n, w^n)
+k^\alpha\tau_\alpha(\!theta^n, \!u^n, w^n)\right] \hfill \\
\hfill -\sum_{\tau\in\T_{h_n}}\left[
\int_{\partial\t\tau}n^{\alpha\beta}n_{\beta}\theta^n_{\alpha}+
\int_{\partial\t\tau}
\left(m^{\alpha\beta}n_{\beta}-n^{\lambda\gamma}b^{\alpha}_{\lambda}n_{\gamma}
\right)u^n_{\alpha}+
\int_{\partial\t\tau}k^\alpha n_{\alpha}w^n\right].
\end{multline*}
In view of the boundary condition in \eqref{m-n-equilibrium}, the second line in the above 
right hand side is equal to
\begin{equation*}
-\sum_{e\in\E^0_{h_n}}\left[\int_{\t e}
n^{\alpha\beta}\lbra \theta^n_\alpha\rbra_{n_\beta}+\int_{\t e}
\left(m^{\alpha\beta}-n^{\lambda\beta}b^{\alpha}_{\lambda}
\right)\lbra u^n_{\alpha}\rbra_{n_\beta}+
\int_{\t e}k^\alpha\lbra w^n\rbra_{n_\alpha}\right].
\end{equation*}
We apply the trace estimate \eqref{trace} to each of the edges, and use Cauchy--Schwarz inequality,  to obtain
the following estimate.
\begin{multline*}
\left|\int_{\t\Omega_{h_n}}(Q^\alpha u^n_{\alpha}+Q^3w^n)\right|
\le C\ll(\!theta^n, \!u^n, w^n)\ll_{h_n}\\
\left[\sum_{\tau\in\T_{h_n}; \alpha, \beta=1,2}\left(| n^{\alpha\beta}|^2_{0,\tau}
+h^2_\tau|n^{\alpha\beta}|^2_{1,\tau}
+| m^{\alpha\beta}|^2_{0,\tau}+h^2_\tau| m^{\alpha\beta}|^2_{1,\tau}+
|k^{\alpha}|^2_{0,\tau}+h^2_\tau|k^{\alpha}|^2_{1,\tau}
\right)
\right]^{1/2}.
\end{multline*}
Since $Q^i$, $m^{\alpha\beta}$, $n^{\alpha\beta}$, and $k^\alpha$  are independent of $n$,
and $\lim_{n\to\infty}\ll(\!theta^n, \!u^n, w^n)\ll_{h_n}=0$, we have
$\int_{\t\Omega}(Q^\alpha u^0_{\alpha}+Q^3w^0)=0$. Thus $u^0_\alpha\!a^\alpha+w^0\!a^3$ is a rigid body motion 
of the shell midsurface.  This, together with \eqref{shear=0}, shows that $(\!theta^0, \!u^0, w^0)\in RBM$.

Finally, we show that $(\!theta^0, \!u^0, w^0)=0$.
It follows from \eqref{Korn-thm-proof0} and $\ll(\!theta^n,  \!u^n, w^n)\ll_{h_n}\to 0$ that
$\lim_{n\to\infty}\|(\!theta^n-\!theta^0, \!u^n-\!u^0, w^n-w^0)\|_{\!H^1_h\x\!H^1_h\x H^1_h}=0$. Since $f$ is uniformly continuous
with respect to the norm $\|\cdot\|_{\!H^1\x\!H^1_h\x H^1_h}$ and since $f(\!theta^n, \!u^n, w^n)\to 0$ ($f$ is a part
in the triple norm), we see
$f(\!theta^0, \!u^0, w^0)=0$. Thus $(\!theta^0, \!u^0, w^0)=0$. Therefore, 
\begin{equation*}
\lim_{n\to\infty}\|(\!theta^n, \!u^n, w^n)\|_{\!H^1_h\x\!H^1_h\x H^1_h}=0,
\end{equation*}
which is contradict to the assumption that 
$\|(\!theta^n, \!u^n, w^n)\|_{\!H^1_h\x\!H^1_h\x H^1_h}=1$.
\end{proof}


As an example, we take
\begin{equation*}
f(\!theta, \!u, w)=
\left[\sum_{e\in\E^D_h}\int_{e}\sum_{\alpha=1,2}\theta^2_\alpha+\sum_{e\in\E^{S}_h\cup\E^D_h}
\left(\int_{e}\sum_{\alpha=1,2}u^2_{\alpha}+\int_{e}w^2\right)
\right]^{1/2}.
\end{equation*}
It follows from Lemma~\ref{tracetheorem} that there is a $C$ only dependent on $\K$ such that
the continuity condition \eqref{f-continuous} is satisfied by this $f$.
Under the assumption that  the measure of $\partial^D\Omega$ is positive,
it is verified in \cite{BCM} that if $(\!theta, \!u, w)\in RBM$ and  $f(\!theta, \!u, w)=0$ then $\!theta=0$, $\!u=0$, and $w=0$.
With this $f$ in the Korn's inequality \eqref{Korn-inequality}, 
we add boundary penalty term
\begin{equation*}
\sum_{e\in\E^D_h}\int_{e}h^{-1}_e\sum_{\alpha=1,2}\theta^2_\alpha+\sum_{e\in\E^{S}_h\cup\E^D_h}
\left(h^{-1}_e\int_{e}\sum_{\alpha=1,2}u^2_{\alpha}+h^{-1}_e\int_{e}w^2\right)
\end{equation*}
to the squares of both sides of \eqref{Korn-inequality}. We then have the equivalence that there is a constant $C$ that could be dependent
on the shell midsurface and shape regularity $\K$ of the triangulation $\T_h$, but otherwise independent of the triangulation, such that
\begin{equation}\label{Hh-ah-equiv}
C^{-1}\|(\!theta, \!u, w)\|_{a_h}\le \|(\!theta,\!u, w)\|_{\H_h}\le C\|(\!theta, \!u, w)\|_{a_h} \ \forall\ (\!theta, \!u, w) \in \!H^1_h\x\!H^1_h\x H^1_h.
\end{equation}
Here
\begin{multline}\label{Hh-norm}
\|(\!theta, \!u, w)\|^2_{\H_h}:=\sum_{\tau\in\T_h}\left[\sum_{\alpha=1,2}\left(\|\theta_\alpha\|^2_{1,\tau}+\|u_\alpha\|^2_{1,\tau}\right)+\|w\|^2_{1,\tau}\right]\\
+\sum_{e\in \E^0_h}h^{-1}_e
\int_{e}\left[\sum_{\alpha=1,2}\left(\lbra \theta_\alpha\rbra^2+\lbra u_\alpha\rbra^2\right)+
\lbra w\rbra^2\right]\\
+
\sum_{e\in\E^{S}_h\cup\E^D_h}
h^{-1}_e\int_{e}\left(\sum_{\alpha=1,2}u^2_{\alpha}+w^2\right)
+\sum_{e\in\E^D_h}
h^{-1}_e\int_{e}\sum_{\alpha=1,2}\theta_\alpha^2,
\end{multline}
\begin{multline}\label{ah-norm}
\|(\!theta,\!u, w)\|^2_{a_h}:=
\sum_{\tau\in\T_h}\left[\sum_{\alpha,\beta=1,2}\left(\|\rho_{\alpha\beta}(\!theta, \!u, w)\|^2_{0,\tau}+\|\gamma_{\alpha\beta}(\!u, w)\|^2_{0,\tau}\right)
+\sum_{\alpha=1,2}\|\tau_\alpha(\!theta, \!u, w)\|^2_{0,\tau}
\right]\\
+\sum_{e\in \E^0_h}h^{-1}_e
\int_{e}\left[\sum_{\alpha=1,2}\left(\lbra \theta_\alpha\rbra^2+\lbra u_\alpha\rbra^2\right)+
\lbra w\rbra^2\right]\\
+
\sum_{e\in\E^{S}_h\cup\E^D_h}
h^{-1}_e\int_{e}\left(\sum_{\alpha=1,2}u^2_{\alpha}+w^2\right)
+\sum_{e\in\E^D_h}
h^{-1}_e\int_{e}\sum_{\alpha=1,2}\theta_\alpha^2.
\end{multline}

\section{Error analysis for  the finite element method}
\label{ErrorAnalysis}
The finite element model defined by \eqref{form_a} to \eqref{N-fem} fits in the form of the mixed equation
\eqref{isomorphism-abs}. We verify the conditions \eqref{isomorphism-condition} for the bilinear forms
defined by \eqref{form_a}, \eqref{form_b}, and \eqref{form_c}, with the space defined by \eqref{FE-space},
in which the $\H_h$ norm is defined by \eqref{Hh-norm}. We define the $\V_h$ norm by
\begin{equation}\label{Vh-norm}
\|(\N, \!eta)\|_{\V_h}:=\left(\sum_{\alpha,\beta=1,2}\|\N^{\alpha\beta}\|^2_{0,\Omega}+ \sum_{\alpha=1,2}\|\eta^\alpha\|^2_{0,\Omega}\right)^{1/2}.
\end{equation}
We show that there is a constant $C$ that depends on the shell geometry and shape regularity $\K$ of the triangulation $\T_h$, but otherwise, independent of the
triangulation such that
\begin{align} 
|a(\!theta, \!u, w; \!phi, \!v, z)| &\le C\|(\!theta, \!u, w)\|_{\H_h}\|(\!phi, \!v, z)\|_{\H_h}&\  &\forall\ (\!theta, \!u, w), (\!phi, \!v, z)\in \H_h,   \label{a-condition1}      \\
\|(\!phi, \!v, z)\|_{\H_h}^2& \le Ca(\!phi, \!v, z; \!phi, \!v, z)&\  &\forall\  (\!phi, \!v, z)\in \H_h,\label{a-condition2} \\
|b(\N, \!eta; \!phi, \!v, z)| &\le C\|(\!phi, \!v, z)\|_{\H_h}\|(\N, \!eta)\|_{\V_h}&\  &\forall\ (\!phi, \!v, z)\in \H_h, (\N, \!eta) \in \V_h,\label{b-condition}\\
|c(\M, \!xi;  \N, \!eta)|&\le C\|(\M, \!xi)\|_{V_h}\|(\N, \!eta)\|_{\V_h}&\  &\forall\ (\M, \!xi), (\N, \!eta) \in \V_h,\label{c-condition1}\\
\|(\N, \!eta)\|_{\V_h}^2&\le Cc(\N, \!eta;  \N, \!eta)&\  &\forall\ (\N, \!eta) \in \V_h. \label{c-condition2}
\end{align} 
We start with \eqref{a-condition1}.
From the definition \eqref{form_ub_a} and \eqref{form_a} of the bilinear form $a$, using the property of the elastic tensor \eqref{elastic-tensor-equiv},
we see the first line
in \eqref{form_ub_a} is bounded as
\begin{multline*}
\int_{\t\Omega_h}\left|\left[a^{\alpha\beta\lambda\gamma}
\rho_{\lambda\gamma}(\!theta, \!u, w)\rho_{\alpha\beta}(\!phi, \!v, z)
+a^{\alpha\beta\lambda\gamma}
\gamma_{\lambda\gamma}(\!u, w)\gamma_{\alpha\beta}(\!v, z)\right.\right.\\
\left.\left. +\kappa\mu a^{\alpha\beta}\tau_\alpha(\!theta, \!u, w)\tau_{\beta}(\!phi, \!v, z)\right]\right|
\le C\|(\!theta, \!u, w)\|_{a_h}\|(\!phi, \!v, z)\|_{a_h}\\
\le C \|(\!theta, \!u, w)\|_{\H_h}\|(\!phi, \!v, z)\|_{\H_h}.
\end{multline*}
We then estimate the second line in \eqref{form_ub_a}. We take on the first term
and let $e\in\E^0_h$ be one of the interior edges shared by elements $\tau_1$ and $\tau_2$. Using the elastic tensor property \eqref{elastic-tensor-equiv},
the H\"older inequality, we have
\begin{multline*}
\left|\int_{\t e}
a^{\alpha\beta\lambda\gamma}\lbrac\rho_{\lambda\gamma}(\!phi, \!v, z)\rbrac\lbra\theta_\alpha\rbra_{n_\beta}\right|\le C
\left[\sum_{\lambda, \gamma=1,2}h_e\int_e\lbrac\rho_{\lambda\gamma}(\!phi, \!v, z)\rbrac^2
\right]^{1/2}
\left[h^{-1}_e\sum_{\alpha=1,2}\int_e\lbra \theta_\alpha\rbra^2
\right]^{1/2}.
\end{multline*}
Using the trace inequality \eqref{trace},  the formula \eqref{N-bending} and \eqref{covariant-derivative}, and
the inverse inequality for finite element functions, we get
\begin{multline*}
h_e\int_e\lbrac\rho_{\lambda\gamma}(\!phi, \!v, z)\rbrac^2
\le C
\sum_{\beta, \delta=1,2}\left(\int_{\tau_\delta}|\rho_{\lambda\gamma}(\!phi, \!v, z)|^2+
h^2_{\tau_\delta}\int_{\tau_\delta}|\partial_\beta\rho_{\lambda\gamma}(\!phi, \!v, z)|^2\right)\\
\le C
\sum_{\delta=1,2}\left[
\|\!phi\|^2_{1,\tau_\delta}+\|\!v\|^2_{1,\tau_\delta}+\|z\|^2_{0,\tau_\delta}+
h^2_{\tau_\delta}
\left(\|\!phi\|^2_{2,\tau_\delta}+\|\!v\|^2_{2,\tau_\delta}+\|z\|^2_{1,\tau_\delta}\right)
\right]
\\
\le C
\sum_{\delta=1,2}\left(
\|\!phi\|^2_{1,\tau_\delta}+\|\!v\|^2_{1,\tau_\delta}+\|z\|^2_{0,\tau_\delta}.
\right)
\end{multline*}
Summing the above estimates for all $e\in\E^0_h$, and using Cauchy--Schwarz inequality, we get
\begin{multline}\label{a-5-est}
\left|\int_{\t\E^0_h}a^{\alpha\beta\lambda\gamma}\lbrac\rho_{\lambda\gamma}(\!phi, \!v, z)\rbrac\lbra\theta_\alpha\rbra_{n_\beta}\right|\\
\le C
\left[\sum_{\tau\in\T_h}\left(\|\!phi\|^2_{1,\tau}+\|\!v\|^2_{1,\tau}+\|z\|^2_{0,\tau}\right)\right]^{1/2}
\left(\sum_{e\in\E^0_h}h^{-1}_e\sum_{\alpha=1,2}\int_e\lbra \theta_\alpha\rbra^2\right)^{1/2}.
\end{multline}
All the other terms in the expression  \eqref{form_ub_a} of the bilinear form of $\ub a$  can be estimated similarly. 
The estimates on the additional penalty terms in the expression \eqref{form_a} 
of the bilinear form $a$ can be obtained by using the Cauchy--Schwarz again.
This completes the proof of \eqref{a-condition1}.

Next, we consider \eqref{a-condition2}.
Let $B(\!theta, \!u, w; \!phi, \!v, z)$ be the  bilinear form
defined by the sum of all the lines but the first one in the definition \eqref{form_ub_a} of $\ub a(\!theta, \!u, w; \!phi, \!v, z)$.
In view of the equivalence \eqref{Hh-ah-equiv}, there are constants $C_1$ and $C_3$ that depend on the shell mid-surface
and $\K$, and $C_2$ that depend on the penalty constant $\C$ in \eqref{form_a} such that
\begin{multline*}
a(\!phi, \!v, z; \!phi, \!v, z)\ge
C_1\|(\!phi, \!v, z)\|^2_{\H_h}\\
+C_2
\left[\sum_{e\in \E^0_h}\left(
\sum_{\alpha=1,2}h^{-1}_e\int_{e}\lbra v_\alpha\rbra^2+\sum_{\alpha=1, 2}
h^{-1}_e\int_{e}\lbra \phi_\alpha\rbra^2
+
h^{-1}_e\int_{e}\lbra z\rbra^2\right)\right.\\
\left.
+
\sum_{e\in\E^{S}_h\cup\E^D_h}
\left(\sum_{\alpha=1,2}h^{-1}_e\int_{e}v^2_{\alpha}+h^{-1}_e\int_{e}z^2\right)
+\sum_{e\in\E^D_h}
\sum_{\alpha=1,2}h^{-1}_e\int_{e}\phi_\alpha^2\right]
-
C_3|B(\!phi, \!v, z; \!phi, \!v, z)|
\end{multline*}
Using the same argument as what used in deriving \eqref{a-5-est}, we have an upper bound that
\begin{multline*}
|B(\!phi, \!v, z; \!phi, \!v, z)|\\
\le C
\|(\!phi, \!v, z)\|_{\H_h}
\left[\sum_{e\in \E^0_h}\left(\sum_{\alpha=1,2}
h^{-1}_e\int_{e}\lbra v_\alpha\rbra^2+\sum_{\alpha=1, 2}
h^{-1}_e\int_{e}\lbra \phi_\alpha\rbra^2
+
h^{-1}_e\int_{e}\lbra z\rbra^2\right)\right.\\
\left.
+
\sum_{e\in\E^{S}_h\cup\E^D_h}
\left(\sum_{\alpha=1,2}h^{-1}_e\int_{e}v^2_{\alpha}+h^{-1}_e\int_{e}z^2\right)
+\sum_{e\in\E^D_h}
\sum_{\alpha=1,2}h^{-1}_e\int_{e}\phi_\alpha^2
\right]^{1/2}.
\end{multline*}
It follows from Cauchy--Schwarz inequality that when the penalty constant $\C$ in \eqref{form_a} is sufficiently big (which makes
$C_2$ sufficiently big)
there is a $C$ such that \eqref{a-condition2} holds.

To see the the continuity \eqref{b-condition} of the bilinear form $b$, we only need to look at the second term 
in the right hand side of \eqref{form_b}. For an $e\in\E^0_h$ shared by $\tau_1$ and $\tau_2$, we have
\begin{multline*}
\left|\int_e\lbrac\N^{\alpha\beta}\rbrac\lbra v_{\alpha}\rbra_{n_{\beta}}\right|
\le C
\left[\sum_{\alpha,\beta=1,2}h_e\int_e(\N^{\alpha\beta})^2
\right]^{1/2}
\left[\sum_{\alpha=1,2}h^{-1}_e\int_e\lbra v_{\alpha}\rbra^2
\right]^{1/2}
\\
\hfill \le C
\left[\sum_{\alpha,\beta,\delta=1,2}(|\N^{\alpha\beta}|^2_{0,\tau_\delta}+h^2_{\tau_\delta}|\N^{\alpha\beta}|^2_{1,\tau_\delta})
\right]^{1/2}
\left[\sum_{\alpha=1,2}h^{-1}_e\int_e\lbra v_{\alpha}\rbra^2
\right]^{1/2}.
\end{multline*}
Similarly, 
\begin{multline*}
\left|\int_e\lbrac\eta^\alpha\rbrac\lbra z\rbra_{n_\alpha}\right|
\le C
\left[\sum_{\alpha,\delta=1,2}(|\eta^{\alpha}|^2_{0,\tau_\delta}+h^2_{\tau_\delta}|\eta^{\alpha}|^2_{1,\tau_\delta})
\right]^{1/2}
\left[h^{-1}_e\int_e\lbra z\rbra^2
\right]^{1/2}.\hfill
\end{multline*}
Summing up these estimates for all $e\in\E^0_h$, and using inverse inequality to the finite element functions $\N$ and $\!eta$, we
obtain
\begin{equation*}
\left|\int_{\t\E^0_h}\left(\lbrac\N^{\alpha\beta}\rbrac\lbra v_{\alpha}\rbra_{n_{\beta}}+\lbrac\eta^\alpha\rbrac\lbra z\rbra_{n_\alpha}\right)\right|
\le C
\|(\N, \!eta)\|_{\V_h}\|(\!phi, \!v, z)\|_{\H_h}.
\end{equation*}

The conditions \eqref{c-condition1} and \eqref{c-condition2} are trivial consequences of \eqref{compliance-tensor-equiv}.

Thus the finite element model \eqref{N-fem} has a unique solution
in the finite element space \eqref{FE-space}.
Corresponding to the weak norm \eqref{isomorphism-weak}, we define a weaker (semi) $\overline\V_h$ norm for finite element function in $\V_h$ by
\begin{equation}\label{isomorphism-weak-N}
|(\N, \!eta)|_{\overline\V_h}:=\sup_{(\!phi, \!v, z)\in \H_h}\frac{b(\N, \!eta; \!phi, \!v, z)}{\|(\!phi, \!v, z)\|_{\H_h}}\ \ \forall\ (\N, \!eta) \in \V_h.
\end{equation}
We are now in a situation in  which Theorem~\ref{isomorphism-thm} is applicable. From that theorem,
we have the inequality that there exists a $C$ that could be dependent on the shell mid-surface and the shape regularity $\K$
of the triangulation $\T_h$, but otherwise independent
of the finite element mesh and the shell thickness $\eps$ such that
\begin{multline*}
\|(\!theta, \!u, w)\|_{\H_h}+|(\M, \!xi)|_{\overline\V_h}+\eps\|(\M, \!xi)\|_{\V_h}\\
\le C
\sup_{(\!phi, \!v, z)\in\H_h, (\N, \!eta)\in\V_h}
\frac{
\left[
\begin{gathered}
a(\!theta, \!u, w; \!phi, \!v, z)+b(\M, \!xi; \!phi, \!v, z)\\
-b(\N, \!eta; \!theta, \!u, w)+\eps^2
c(\M, \!xi; \N, \!eta)
\end{gathered}
\right]
} {\|(\!phi, \!v, z)\|_{\H_h}+|(\N, \!eta)|_{\overline\V_h}+\eps\|(\N, \!eta)\|_{\V_h}}\\
 \ \forall\
(\!theta, \!u, w)\in\H_h, (\M, \!xi)\in\V_h.
\end{multline*}
Let $\!theta\e, \!u\e, w\e,\M\e, \!xi\e$ be the solution to the mixed formulation of the Naghdi  model \eqref{N-P-model}, let
$\!theta^h, \!u^h, w^h, \M^h, \!xi^h$ be the finite element solution to the finite element model \eqref{N-fem},
and let $\!theta^I, \!u^I, w^I, \M^I, \!xi^I$ be an interpolation to the Naghdi  model solution from the finite element space.
Since the finite element method \eqref{N-fem} and the Naghdi model \eqref{N-P-model}
are consistent, we have
\begin{multline}\label{error-fraction}
C^{-1}\|(\!theta^h-\!theta^I, \!u^h-\!u^I, w^h-w^I)\|_{\H_h}\\+
C^{-1}
\left[|(\M^h-\M^I, \!xi^h-\!xi^I)|_{\overline\V_h}+\eps\|(\M^h-\M^I, \!xi^h-\!xi^I)\|_{\V_h}\right]
\le\\
\sup_{
(\!phi, \!v, z)\in\H_h, (\N,\!eta)\in\V_h}
\frac{\left[\begin{gathered}
a(\!theta^h-\!theta^I, \!u^h-\!u^I, w^h-w^I; \!phi, \!v, z)+b(\M^h-\M^I, \!xi^h-\!xi^I; \!phi, \!v, z)\\
-b(\N, \!eta; \!theta^h-\!theta^I, \!u^h-\!u^I, w^h-w^I)+\eps^2
c(\M^h-\M^I, \!xi^h-\!xi^I, \N, \!eta)\end{gathered}
\right]} {\|(\!phi, \!v, z)\|_{\H_h}+|(\N, \!eta)|_{\overline\V_h}+\eps\|(\N, \!eta)\|_{\V_h}}\\
=
\sup_{
(\!phi, \!v, z)\in\H_h, (\N,\!eta)\in\V_h}
\frac{\left[\begin{gathered}
a(\!theta\e-\!theta^I, \!u\e-\!u^I, w\e-w^I; \!phi, \!v, z)+b(\M\e-\M^I, \!xi\e-\!xi^I; \!phi, \!v, z)\\
-b(\N, \!eta; \!theta\e-\!theta^I, \!u\e-\!u^I, w\e-w^I)+\eps^2
c(\M\e-\M^I, \!xi\e-\!xi^I, \N, \!eta)\end{gathered}
\right]} {\|(\!phi, \!v, z)\|_{\H_h}+|(\N, \!eta)|_{\overline\V_h}+\eps\|(\N, \!eta)\|_{\V_h}}.
\end{multline}
We estimate the four terms in the numerator of above last line one by one.
\begin{lem}\label{a-error-lem}
There is a constant $C$ independent of $\T_h$ such that
\begin{multline}\label{a-error}
\left|a(\!theta\e-\!theta^I, \!u\e-\!u^I, w\e-w^I;\!phi, \!v, z)\right|
\\
\le C\left\{
\sum_{\tau\in\T_h}\left[
\sum_{k=0}^2h^{2k-2}_\tau\left( \sum_{\alpha=1,2}|\theta\e_\alpha-\theta^I_\alpha|^2_{k,\tau}+\sum_{\alpha=1,2}|u\e_\alpha-u^I_\alpha|^2_{k,\tau}
+|w\e-w^I|_{k,\tau}
\right)\right]
\right\}^{1/2}\\
\|(\!phi, \!v, z)\|_{\H_h}\ \forall\ (\!phi, \!v, z)\in \H_h.
\end{multline}
\end{lem}
\begin{proof}
There are many terms in the expression of  
$a(\!theta\e-\!theta^I, \!u\e-\!u^I, w\e-w^I;\!phi, \!v, z)$, see \eqref{form_a} and \eqref{form_ub_a}. 
We only present  estimations for a few typical terms. The others can be bounded in similar ways. It is easy to see
\begin{multline*}
\left|\int_{\t\Omega_h}
a^{\alpha\beta\lambda\gamma}\rho_{\lambda\gamma}(\!theta\e-\!theta^I, \!u\e-\!u^I, w\e-w^I)
\rho_{\alpha\beta}(\!phi, \!v, z)\right|\\
\le C
\left\{\sum_{\tau\in\T_h}\left[
\left( \sum_{\alpha=1,2}\|\theta\e_\alpha-\theta^I_\alpha\|^2_{1,\tau}+\sum_{\alpha=1,2}\|u\e_\alpha-u^I_\alpha\|^2_{1,\tau}
+|w\e-w^I|_{0,\tau}
\right)\right]
\right\}^{1/2}\|(\!phi, \!v, z)\|_{\H_h}.
\end{multline*}
For an edge $e\in\E^0_h$ shared by $\tau_1$ and $\tau_2$, we have 
\begin{multline*}
\left|\int_e\lbrac\rho_{\lambda\gamma}(\!phi, \!v, z)\rbrac\lbra\theta\e_\alpha-\theta^I_\alpha \rbra_{n_\beta}\right|
\le C
\left(
h_e\int_e\lbrac\rho_{\lambda\gamma}(\!phi, \!v, z)\rbrac^2
\right)^{1/2}
\left(
h^{-1}_e\int_e\lbra\theta\e_\alpha-\theta^I_\alpha \rbra^2
\right)^{1/2}\\
\le C
\left[
\sum_{\delta=1,2}\left(
|\rho_{\lambda\gamma}(\!phi, \!v, z)|^2_{0,\tau_\delta}+h^2_{\tau_\delta}|\rho_{\lambda\gamma}(\!phi, \!v, z)|^2_{1,\tau_\delta}\right)
\right]^{1/2}
\left(
h^{-1}_e\int_e\lbra\theta\e_\alpha-\theta^I_\alpha \rbra^2
\right)^{1/2}
\\
\le C
\left[
\sum_{\delta=1,2}\left(
\|\!phi\|^2_{1,\tau_\delta}+\|\!v\|^2_{1,\tau_\delta}+\|z\|^2_{0,\tau_\delta}
\right)
\right]^{1/2}
\left[
\sum_{\delta=1,2}\left(h^{-2}_{\tau_\delta}|\theta\e_\alpha-\theta^I_\alpha|^2_{0,\tau_\delta}
+|\theta\e_\alpha-\theta^I_\alpha|^2_{1,\tau_\delta}\right)
\right]^{1/2},
\end{multline*}
\begin{multline*}
\left|\int_{e}\lbrac\rho_{\lambda\gamma}(\!theta\e-\!theta^I, \!u\e-\!u^I, w\e-w^I)\rbrac\lbra\phi_\alpha\rbra_{n_\beta}\right|
\\
\le C
\left[
\sum_{\delta=1,2}\left(
\|\!theta\e-\!theta^I\|^2_{1,\tau_\delta}+\|\!u\e-\!u^I\|^2_{1,\tau_\delta}+\|w\e-w^I\|^2_{0,\tau_\delta}\right)\right.
\\
+\left.
\sum_{\delta=1,2}h^2_{\tau_\delta}\left(|\!theta\e-\!theta^I|^2_{2,\tau_\delta}+|\!u\e-\!u^I|^2_{2,\tau_\delta}+|w\e-w^I|^2_{1,\tau_\delta}
\right)
\right]^{1/2}
\left(
h^{-1}_e\int_e\lbra\phi_\alpha \rbra^2
\right)^{1/2}.
\end{multline*}

For an edge $e\in \E^{D\cup S}_h$ that is an edge of element $\tau$, we have 
\begin{multline*}
\left|\int_{e}\tau_\beta(\!phi, \!v, z)(w\e-w^I)\right|\le C
\left[h_e\int_e|\tau_\beta(\!phi, \!v, z)|^2
\right]^{1/2}
\left[(h^{-1}_e\int_e(w\e-w^I)^2
\right]^{1/2}\\
\le C
\left[|\!phi|^2_{0,\tau}+|\!v|^2_{0,\tau}+\|z\|^2_{1,\tau}
\right]^{1/2}
\left[h^{-2}_\tau|w\e-w^I|^2_{0,\tau}+|w\e-w^I|^2_{1,\tau}
\right]^{1/2},
\end{multline*}
\begin{multline*}
\left|\int_e\tau_\beta(\!theta\e-\!theta^I, \!u\e-\!u^I, w\e-w^I)z\right|
\le C
\left[|\!theta\e-\!theta^I|^2_{0,\tau}+|\!u\e-\!u^I|^2_{0,\tau}+\|w\e-w^I\|^2_{1,\tau}\right.\\
\left.+h^2_\tau\left(|\!theta\e-\!theta^I|^2_{1,\tau}+|\!u\e-\!u^I|^2_{1,\tau}+|w\e-w^I|^2_{2,\tau}\right)
\right]^{1/2}
\left(h^{-1}_e\int_ez^2
\right)^{1/2}.
\end{multline*}

As  a typical penalty term in \eqref{form_a}, we consider an $e\in\E^0_h$ shared by $\tau_1$ and $\tau_2$, and we have
\begin{multline*}
\left|h^{-1}_e\int_{e}\lbra u\e_{\alpha}-u^I_\alpha\rbra \lbra v_{\alpha}\rbra\right|
\le C
\left[h^{-1}_e\int_e\lbra u\e_{\alpha}-u^I_\alpha\rbra^2
\right]^{1/2}
\left[h^{-1}_e\int_e\lbra v_\alpha\rbra^2
\right]^{1/2}\\
\le C
\left[\sum_{\delta=1,2}\left(h^{-2}_{\tau_\delta}|u\e_{\alpha}-u^I_\alpha|^2_{0,\tau_\delta}+|u\e_{\alpha}-u^I_\alpha|^2_{1,\tau_\delta}\right)
\right]^{1/2}
\left[h^{-1}_e\int_e\lbra v_\alpha\rbra^2
\right]^{1/2}.
\end{multline*}
Any single term in the expression of $a(\!theta\e-\!theta^I, \!u\e-\!u^I, w\e-w^I;\!phi, \!v, z)$ can be estimated in 
a way used above. The desired result then follows from summing up estimates for all the term and using Cauchy--Schwarz inequality.
\end{proof}

\begin{lem}\label{b2-error-lem}
There is a $C$ independent of $\T_h$ such that
\begin{multline}\label{b2-error}
\left|b(\M\e-\M^I, \!xi\e-\!xi^I; \!phi, \!v,z)\right|\\
\le C
\left\{\sum_{\tau\in\T_h}
\left[\sum_{\alpha,\beta=1,2}\left(|\M^{\eps\alpha\beta}-\M^{I \alpha\beta}|^2_{0, \tau}+h^2_{\tau}|\M^{\eps\alpha\beta}-\M^{I\alpha\beta}|^2_{1, \tau}\right)\right.\right.\\
\left.\left.
+
\sum_{\alpha=1,2}\left(|\xi^{\eps\alpha}-\xi^{I \alpha}|^2_{0, \tau}+h^2_{\tau}|\xi^{\eps\alpha}-\xi^{I\alpha}|^2_{1, \tau}
\right)\right]
\right\}^{1/2}
\|(\!phi, \!v, z)\|_{\H_h}\ \forall\ (\!phi, \!v, z)\in\H_h.
\end{multline}
\end{lem}
\begin{proof}
In view of the definition \eqref{form_b}, we have
\begin{multline*} 
b(\M\e-\M^I,\!xi\e-\!xi^I; \!phi, \!v,z)=
\int_{\t\Omega_h}\left[(\M^{\eps\alpha\beta}-\M^{I\alpha\beta})
\gamma_{\alpha\beta}(\!v, z)+(\xi^{\eps\alpha}-\xi^{I\alpha})\tau_\alpha(\!phi, \!v, z)\right]\\
-\int_{\t\E^0_h}\left(\lbrac\M^{\eps\alpha\beta}-\M^{I\alpha\beta}\rbrac\lbra v_{\alpha}\rbra_{n_{\beta}}+\lbrac\xi^{\eps\alpha}-\xi^{I\alpha}\rbrac\lbra z\rbra_{n_\alpha}\right)\\
 -\int_{\t\E^{D\cup S}_h}\left[
(\M^{\eps\alpha\beta}-\M^{I\alpha\beta}){n_{\beta}}v_{\alpha}+(\xi^{\eps\alpha}-\xi^{I\alpha}) n_\alpha z\right],
\end{multline*}
We have the estimates on the $\M$ related terms
\begin{equation*}
\left|\int_{\t\Omega_h}(\M^{\eps\alpha\beta}-\M^{I\alpha\beta})\gamma_{\alpha\beta}(\!v, z)\right|
\le C
\sum_{\alpha, \beta=1,2}|\M^{\eps\alpha\beta}-\M^{I\alpha\beta}|_{0, \Omega_h}\sum_{\alpha, \beta=1,2}|\gamma_{\alpha\beta}(\!v, z)|_{0, \Omega_h},
\end{equation*}
\begin{multline*}
\left|\int_{\t\E^0_h}\lbrac\M^{\eps\alpha\beta}-\M^{I\alpha\beta}\rbrac\lbra v_{\alpha}\rbra_{n_{\beta}}\right|\\
\le C
\sum_{e\in\E^0_h}\left[\sum_{\alpha, \beta=1,2}h_e|\M^{\eps\alpha\beta}-\M^{I \alpha\beta}|^2_{0, e}\right]^{1/2}
\left[h^{-1}_e|\lbra \!v\rbra|^2_{0,e}\right]^{1/2}\\
\le C
\left[\sum_{\tau\in\T_h}\sum_{\alpha,\beta=1,2}\left(|\M^{\eps\alpha\beta}-\M^{I \alpha\beta}|^2_{0, \tau}+h^2_{\tau}|\M^{\eps\alpha\beta}-\M^{I\alpha\beta}|^2_{1, \tau}\right)
\right]^{1/2}\left[\sum_{e\in\E^0_h}h^{-1}_e|\lbra \!v\rbra|^2_{0,e}\right]^{1/2},
\end{multline*}
and
\begin{multline*}
\left|\int_{\t\E^{S}_h\cup\t\E^D_h}
(\M^{\eps\alpha\beta}-\M^{I\alpha\beta}){n_{\beta}}v_{\alpha}\right|\\
\le C
\sum_{e\in\E^S_h\cup\E^D_h}\left[\sum_{\alpha, \beta=1,2}h_e|\M^{\eps\alpha\beta}-\M^{I \alpha\beta}|^2_{0, e}\right]^{1/2}
\left[h^{-1}_e|\!v|^2_{0,e}\right]^{1/2}\\
\le C
\left[\sum_{\tau\in\T_h}\sum_{\alpha,\beta=1,2}\left(|\M^{\eps\alpha\beta}-\M^{I \alpha\beta}|^2_{0, \tau}+h^2_{\tau}|\M^{\eps\alpha\beta}-\M^{I\alpha\beta}|^2_{1, \tau}\right)
\right]^{1/2}\left[\sum_{e\in\E^S_h\cup\E^D_h}h^{-1}_e|\!v|^2_{0,e}\right]^{1/2}.
\end{multline*}
The $\xi$ related terms can be bounded in the same way.
Summing up, we get the estimate \eqref{b2-error}.
\end{proof}

In either  the inequality \eqref{a-error} or  \eqref{b2-error}, we did not
impose any condition for the interpolations $\!theta^I$, $\!u^I$, $w^I$, $\M^I$, and $\!xi^I$,  except that they
are finite element functions from the space \eqref{FE-space}. The next estimate is very different in that
the interpolation needs to be  particularly chosen to obtain a desirable bound for
\begin{multline}\label{b-original}
b(\N, \!eta; \!theta\e-\!theta^I, \!u\e-\!u^I, w\e-w^I)\\
=
\int_{\t\Omega_h}\left[\N^{\alpha\beta}\gamma_{\alpha\beta}(\!u\e-\!u^I, w\e-w^I)+\eta^\alpha\tau_\alpha(\!theta\e-\!theta^I, \!u\e-\!u^I, w\e-w^I)\right]\\
-\int_{\t\E^0_h}\left(\lbrac\N^{\alpha\beta}\rbrac\lbra u\e_{\alpha}-u^I_\alpha\rbra_{n_{\beta}}+\lbrac\eta^\alpha\rbrac\lbra w\e-w^I\rbra_{n_\alpha}\right)
\\ -\int_{\t\E^{S}_h\cup\t\E^D_h}
\left[\N^{\alpha\beta}{n_{\beta}}(u\e_{\alpha}-u^I_\alpha)+\eta^\alpha n_\alpha (w\e-w^I)\right].
\end{multline}

For $\theta\e_\alpha$, on any $\tau\in\T_h$, 
we define $\theta^I_\alpha\in P^2(\tau)$ by the weighted $L^2(\tau)$ projection such that 
\begin{equation}\label{thetaI}
\int_{\t\tau}(\theta\e_\alpha-\theta^I_\alpha)p=0\ \forall\ p\in P^2(\tau).
\end{equation}

For $u\e_\alpha$ and $w\e$ , on a $\tau\in\T_h$, if $\partial\tau\cap\E^F_h=\emptyset$, we define $u^I_\alpha$ and $w^I$ in $P^2(\tau)$ by
the weighted $L^2(\tau)$ projection such that
\begin{equation}\label{uwI-interior}
\int_{\t\tau}(u\e_\alpha-u^I_\alpha)p=0, \quad 
\int_{\t\tau}(w\e-w^I)p=0\ \forall\ p\in P^2(\tau).
\end{equation}
If $\partial\tau\cap\E^F_h$ has one edge $e$, we define $u^I_\alpha$ and $w^I$ in $P^3_*(\tau)$ by
\begin{equation}\label{uwI-edge1}
\begin{gathered}
\int_{\t\tau}(u\e_\alpha-u^I_\alpha)p=0, \quad
\int_{\t\tau}(w\e-w^I)p=0\ \forall\ p\in P^2(\tau), \\
\int_e(u\e_\alpha-u^I_\alpha) p\sqrt a=0, \quad
\int_e(w\e-w^I) p\sqrt a=0
\ \forall\ p\in P^1(e).
\end{gathered}
\end{equation}
If $\partial\tau\cap\E^F_h$ has two edges $e_\beta$, we define $u^I_\alpha$ and $w^I$ in $P^3(\tau)$ by
\begin{equation}\label{uwI-edge2}
\begin{gathered}
\int_{\t\tau}(u\e_\alpha-u^I_\alpha)p=0, \quad \int_{\t\tau}(w\e-w^I)p=0
\ \forall\ p\in P^2(\tau), \\
\int_{e_\beta}(u\e_\alpha-u^I_\alpha) p\sqrt a=0,\quad 
\int_{e_\beta}(w\e-w^I) p\sqrt a=0
\ \forall\ p\in P^1(e_\beta).
\end{gathered}
\end{equation}
The unisolvences of \eqref{thetaI} and \eqref{uwI-interior} are trivial. The unisolvence of \eqref{uwI-edge1} is seen from the condition
\eqref{P3*}. To see the unisolvence of \eqref{uwI-edge2}, one uses Appell's polynomial \cite{Braess}
to decompose
a cubic polynomial as the sum of a quadratic and an orthogonal complement, and uses the fact that
the orthogonal complement is uniquely determined by its averages and first moments on two edges.

\begin{lem}\label{b3-error-lem}
With the  interpolations defined by \eqref{thetaI} to \eqref{uwI-edge2}, there is a constant $C$ independent of $\T_h$ such that
\begin{multline}\label{b3-error}
|b(\N, \!eta; \!theta\e-\!theta^I,\!u\e-\!u^I,w\e-w^I)|\\
\le C
\max_{\tau\in\T_h}h^3_\tau\left[\sum_{\alpha,\beta,\lambda=1,2}|\Gamma^{\lambda}_{\alpha\beta}|_{2,\infty,\tau}+
\sum_{\alpha,\beta=1,2}\left(|b_{\alpha\beta}|_{2,\infty,\tau}+|b^\beta_\alpha|_{2,\infty,\tau}\right)\right]\\
\|(\N, \!eta)\|_{\V_h}
\left[\sum_{\tau\in\T_h}h^{-2}_{\tau}
\left(\left|u\e_{\alpha}-u^I_{\alpha}\right|^2_{0,\tau}+
\left|w\e-w^I\right|^2_{0,\tau}\right)
\right]^{1/2}\ \ \forall\ (\N, \!eta)\in\V_h.
\end{multline}
\end{lem}
\begin{proof}
With an application of the Green's theorem \eqref{Green} on each element $\t\tau\in\t\T_h$, summing up, we obtain the following  alternative expression to \eqref{b-original}.
\begin{multline*}
b(\N, \!eta; \!theta\e-\!theta^I, \!u\e-\!u^I,w\e-w^I)
=
\int_{\t\Omega_h}\left[-\N^{\alpha\beta}|_{\beta}\left(u\e_{\alpha}-u^I_{\alpha}\right)-b_{\alpha\beta}\N^{\alpha\beta}
\left(w\e-w^I\right) \right.\\
\left. -\eta^\alpha|_\alpha(w\e-w^I)+\eta^\alpha(\theta\e_\alpha-\theta^I_\alpha)+\eta^\beta b^\alpha_\beta(u\e_\alpha-u^I_\alpha)
\right]\\
+\int_{\t\E_h^0}\lbra\N^{\alpha\beta}\rbra_{n_{\beta}}\lbrac u\e_{\alpha}-u^I_{\alpha}\rbrac+
\int_{\t\E_h^0}\lbra\eta^{\alpha}\rbra_{n_\alpha}\lbrac w\e-w^I\rbrac\\
+\int_{\t\E_h^F}\N^{\alpha\beta}n_{\beta}
\left(u\e_{\alpha}-u^I_{\alpha}\right)+\int_{\t\E_h^F}\eta^\alpha n_\alpha(w\e-w^I).
\end{multline*}
Since $\N^{\alpha\beta}$ and $\eta^\alpha$ are continuous piecewise linear polynomials, on each $e\in\E^0_h$ we have $\lbra\N^{\alpha\beta}\rbra_{n_{\beta}}=0$
and $\lbra\eta^{\alpha}\rbra_{n_\alpha}=0$.
For each $e\in\E^F_h$, we have, see \eqref{Green},
\begin{equation*}
\begin{gathered}
\int_{\t e}\N^{\alpha\beta}n_{\beta}
\left(u\e_{\alpha}-u^I_{\alpha}\right)=\int_{e}\N^{\alpha\beta}\bar n_{\beta}
\left(u\e_{\alpha}-u^I_{\alpha}\right)\sqrt a=0,\\
\int_{\t e}\eta^\alpha n_\alpha(w\e-w^I)=\int_e\eta^\alpha \bar n_\alpha(w\e-w^I)\sqrt a=0.
\end{gathered}
\end{equation*}
Using the formulas \eqref{covariant-derivative} for the covariant derivatives  $\N^{\alpha\beta}|_{\beta}$ and $\eta^\alpha|_\alpha$, and using properties 
of the interpolations \eqref{thetaI} to \eqref{uwI-edge2}, the expression is further simplified to
\begin{multline}\label{b-simple}
b(\N, \!eta; \!theta\e-\!theta^I, \!u\e-\!u^I,w\e-w^I)
\\=
\int_{\t\Omega_h}\left[\left(b^\alpha_\beta\eta^\beta-
\Gamma^{\beta}_{\beta\gamma}\N^{\alpha\gamma}-
\Gamma^{\alpha}_{\delta\beta}\N^{\delta\beta}\right)
\left(u\e_{\alpha}-u^I_{\alpha}\right)-\left(\Gamma^\delta_{\delta\alpha}\eta^\alpha+b_{\alpha\beta}\N^{\alpha\beta}\right)
\left(w\e-w^I\right) \right].
\end{multline}
The last term is estimated as follows. 
For $\tau\in\T_h$, we have
\begin{equation*} 
\int_{\t\tau}b_{\alpha\beta}\N^{\alpha\beta}
\left(w\e-w^I\right)=
\int_{\t\tau}\left[b_{\alpha\beta}-p^1(b_{\alpha\beta})\right]\N^{\alpha\beta}
\left(w\e-w^I\right)
\end{equation*}
Here, $p^1(b_{\alpha\beta})$ is the best linear approximation to $b_{\alpha\beta}$
in the space $L^{\infty}(\tau)$ such that
\begin{equation*}
\left|b_{\alpha\beta}-p^1(b_{\alpha\beta})\right|_{0,\infty,\tau}\le Ch^2_\tau \left|b_{\alpha\beta}\right|_{2,\infty,\tau}.
\end{equation*}
From this, we see
\begin{equation*} 
\left|\int_{\t\tau}b_{\alpha\beta}\N^{\alpha\beta}
\left(w\e-w^I\right)\right|
\le C
h^3_\tau\left|b_{\alpha\beta}\right|_{2,\infty,\tau}|\N^{\alpha\beta}|_{0, \tau}h^{-1}_\tau|w\e-w^I|_{0,\tau}.
\end{equation*}
Summing up such estimates for all $\tau\in\T_h$, and using Cauchy--Schwarz inequality, we get
\begin{multline*}
\left|\int_{\t\Omega_h}b_{\alpha\beta}\N^{\alpha\beta}
\left(w\e-w^I\right)\right|\\
\le C
\left[\max_{\tau\in\T_h}\left(h^3_\tau\sum_{\alpha,\beta=1,2}\left|b_{\alpha\beta}\right|_{2,\infty,\tau}\right)\right]\|\N\|_{0,\Omega_h}
\left(\sum_{\tau\in\T_h}h^{-2}_\tau|w\e-w^I|^2_{0,\tau}
\right)^{1/2}.
\end{multline*}
The other terms in \eqref{b-simple} can be estimated in the same way.
\end{proof}

It is trivial to see that
\begin{multline}\label{c-error}
\left|c(\M\e-\M^I, \!xi\e-\!xi^I; \N, \!eta)\right|\\
\le C \|(\N, \!eta)\|_{\V_h}
\left(\sum_{\alpha,\beta=1,2}|\M^{\eps\alpha\beta}-\M^{I\alpha\beta}|_{0,\Omega_h}
+\sum_{\alpha=1,2}|\xi^{\eps\alpha}-\xi^{I\alpha}|_{0,\Omega_h}
\right)\ \forall\ (\N, \!eta) \in\V_h.
\end{multline}

The following 
result gives an 
estimate for the difference between the finite element solution and an interpolation
of the Naghdi model solution.
It is a result of combining  \eqref{a-error}, \eqref{b2-error}, \eqref{b3-error}, \eqref{c-error}, and \eqref {error-fraction}.
\begin{thm}\label{N-fem-lem}
Let $(\!theta^h, \!u^h, w^h)$ and $(\M^h, \!xi^h)$ be the finite element solution determined by the finite element model 
\eqref{N-fem}. 
Let $\theta^I_\alpha$, $u^I_{\alpha}$, and $w^I$ be the interpolations to $\theta\e_\alpha$, $u\e_\alpha$, and $w\e$
in the finite element space \eqref{FE-space}, which is defined by the formulas \eqref{thetaI}, \eqref{uwI-interior},
\eqref{uwI-edge1}, and \eqref{uwI-edge2}, respectively. Let
$\M^{I\alpha\beta}$ and $\xi^{I\alpha}$ be approximations
to $\M^{\eps\alpha\beta}$ and $\xi^{\eps\alpha}$  to be selected from the space of continuous piecewise linear functions. There is a $C$
independent of the triangulation $\T_h$ and the shell thickness $\eps$ such that
\begin{multline}\label{N-fem-error-lem}
\|(\!theta^h-\!theta^I, \!u^h-\!u^I, w^h-w^I)\|_{\H_h}\\
+
|(\M^h-\M^I, \!xi^h-\!xi^I)|_{\overline\V_h}+\eps\|(\M^h-\M^I, \!xi^h-\!xi^I)\|_{\V_h}
\\
\le C
\left[1+\eps^{-1}
\max_{\tau\in\T_h}h^3_\tau\left(\sum_{\alpha,\beta,\lambda=1,2}|\Gamma^{\lambda}_{\alpha\beta}|_{2,\infty,\tau}+
\sum_{\alpha,\beta=1,2}|b_{\alpha\beta}|_{2,\infty,\tau}+\sum_{\alpha,\beta=1,2}|b^\beta_\alpha|_{2,\infty,\tau}\right)
\right]
\\
\left\{
\sum_{\tau\in\T_h}\left[
\sum_{k=0}^2h^{2k-2}_\tau\left(
\sum_{\alpha=1,2}
\left(|\theta\e_{\alpha}-\theta^I_{\alpha}|^2_{k, \tau}
+
|u\e_{\alpha}-u^I_{\alpha}|^2_{k, \tau}\right)
+
|w\e-w^I|^2_{k, \tau}\right)
\right.\right.\\
\left.\left.+\sum_{k=0}^1h^{2k}_\tau\left(\sum_{\alpha,\beta=1,2}|\M^{\eps\alpha\beta}-\M^{I\alpha\beta}|^2_{k, \tau}
+\sum_{\alpha=1,2}|\xi^{\eps\alpha}-\xi^{I\alpha}|^2_{k,\tau}\right)
\right]\right\}^{1/2}.
\end{multline}
\end{thm}
Although we see an estimate for the weak (semi) norm $|(\M^h-\M^I, \!xi^h-\!xi^I)|_{\overline\V_h}$ from this inequality,
we do not know how to interpret  it.
We therefore can not make any statement on the accuracy of approximating $(\M\e, \!xi\e)$ by the part of the finite element solution
$(\M^h, \!xi^h)$.
We will not pursue this lead, but concentrate on the error estimate for the primary variables.
We assume that  for fixed $\eps$, the Naghdi model solution has the $H^3$ regularity. 
Under this assumption, components of the scaled membrane stress tensor and scaled transverse shear stress vector have the $H^2$ 
regularity. Note that this regularity assumption does not 
imply that the $H^3$ norm of the model primary solution or $H^2$ norm 
of the scaled membrane stress and transverse shear stress  
 are uniformly bounded. Instead, it is very likely that when $\eps\to 0$ these functions would grow unboundedly in these spaces.
We have the following theorem on the error estimate for the finite element method.

\begin{thm}
If the Naghdi model solution has the regularity that $\theta\e_\alpha$,  $u\e_\alpha$,  and $w\e$ have finite $H^3$ norms, then
there is a constant $C$ that is independent of the triangulation $\T_h$ and the shell thickness $\eps$, such that
\begin{multline*}
\|(\!theta\e-\!theta^h, \!u\e-\!u^h, w\e-w^h)\|_{\H_h}\\
\le C
\left[1+\eps^{-1}
\max_{\tau\in\T_h}h^3_\tau\left(\sum_{\alpha,\beta,\lambda=1,2}|\Gamma^{\lambda}_{\alpha\beta}|_{2,\infty,\tau}+
\sum_{\alpha,\beta=1,2}|b_{\alpha\beta}|_{2,\infty,\tau}+\sum_{\alpha,\beta=1,2}|b^\beta_\alpha|_{2,\infty,\tau}\right)
\right]
\\
\left\{
\sum_{\tau\in\T_h}h^4_{\tau}\left[\sum_{\alpha=1,2}\left(\|\theta\e_\alpha\|^2_{3,\tau}+\|u\e_\alpha\|^2_{3,\tau}\right)+\|w\e\|^2_{3,\tau}+
\sum_{\alpha, \beta=1,2}\|\M^{\eps\alpha\beta}\|^2_{2,\tau}+\sum_{\alpha=1,2}\|\xi^{\eps\alpha}\|^2_{2,\tau}
\right]\right\}^{1/2}.
\end{multline*}
Here $(\!theta^h, \!u^h, w^h)$ is the primary part of the solution of 
the finite element model \eqref{N-fem} with the finite element space defined by \eqref{FE-space}. The norm $\|\cdot\|_{\H_h}$
is defined by \eqref{Hh-norm}.
\end{thm}
\begin{proof}
In view of the triangle inequality, we have
\begin{multline*}
\|(\!theta\e-\!theta^h, \!u\e-\!u^h, w\e-w^h)\|_{\H_h}\\
\le \|(\!theta\e-\!theta^I, \!u\e-\!u^I, w\e-w^I)\|_{\H_h}+\|(\!theta^h-\!theta^I, \!u^h-\!u^I, w^h-w^I)\|_{\H_h}.
\end{multline*}
Using the trace inequality \eqref{trace} to the edge terms in the norm  $\|(\!theta\e-\!theta^I, \!u\e-\!u^I, w\e-w^I)\|_{\H_h}$, cf., \eqref{Hh-norm}, we get
\begin{multline*}
\|(\!theta\e-\!theta^I, \!u\e-\!u^I, w\e-w^I)\|_{\H_h}\le C\\
\left\{
\sum_{\tau\in\T_h}\sum_{k=0, 1}h^{2k-2}_\tau\left[
\sum_{\alpha=1,2}\left(|u\e_{\alpha}-u^I_{\alpha}|^2_{k, \tau}+|\theta\e_{\alpha}-\theta^I_{\alpha}|^2_{k, \tau}\right)
+|w\e-w^I|^2_{k, \tau}
\right]\right\}^{1/2}.
\end{multline*}
For each $\tau\in\T_h$, we establish that
\begin{equation}\label{scaled-estimate}
\begin{gathered}
\sum_{k=0}^2h^{2k-2}_\tau|\theta\e_{\alpha}-\theta^I_{\alpha}|^2_{k, \tau}\le C h^4_\tau|\theta\e_\alpha|_{3,\tau},\\
\sum_{k=0}^2h^{2k-2}_\tau|u\e_{\alpha}-u^I_{\alpha}|^2_{k, \tau}\le C h^4_\tau|u\e_\alpha|_{3,\tau},\\
\sum_{k=0}^2h^{2k-2}_\tau|w\e-w^I|^2_{k, \tau}\le C h^4_\tau|w\e|^2_{3,\tau}.
\end{gathered}
\end{equation}

We scale $\tau$ to a similar triangle $\T$ whose diameter is $1$ by the scaling $X_\alpha=h^{-1}_\tau x_\alpha$. Let $\Theta_\beta(X_\alpha)=\theta\e_\beta(x_\alpha)$, 
$U_\beta(X_\alpha)=u\e_\beta(x_\alpha)$, 
$W(X_\alpha)=w\e(x_\alpha)$,
$A(X_\alpha)=a(x_\alpha)$, $\Theta^I_\beta(X_\alpha)=\theta^I_\beta(x_\alpha)$, $U^I_\beta(X_\alpha)=u^I_\beta(x_\alpha)$, and
$W^I(X_\alpha)=w^I(x_\alpha)$. It is easy to see that $\Theta^I_\alpha$ is the projection
of $\Theta_\alpha$ into $P^2(\T)$ in the space $L^2(\T)$ weighted by $\sqrt{A(X_\alpha)}$.
This projection preserves quadratic  polynomials and we have the bound that
\begin{equation}\label{ThetaI-bound}
\|\Theta_\alpha^I\|_{0,\T}\le \left[\frac{\max_{\tau}a}{\min_{\tau}a}\right]^{1/4}\|\Theta_\alpha\|_{0,\T}.
\end{equation}
For a $\Theta_\alpha\in H^3(\T)$ and any quadratic polynomial $p$, using inverse inequality, there is a $C$ depending on the shape regularity of $\T$
such that
\begin{equation*}
\|\Theta_\alpha-\Theta^I_\alpha\|_{3, \T}\le \|\Theta_\alpha-p\|_{3, \T}+\|(\Theta_\alpha-p)^I\|_{3, \T}\le \|\Theta_\alpha-p\|_{3, \T}+C\|(\Theta_\alpha-p)^I\|_{0, \T}.
\end{equation*}
Therefore, there is a $C$ depending on the shape regularity of $\T$ and the  ratio ${\max_{\tau}a}/{\min_{\tau}a}$ such that
\begin{equation*}
\|\Theta_\alpha-\Theta^I_\alpha\|_{3, \T}\le C\|\Theta_\alpha-p\|_{3, \T}\ \ \forall\ p\in P^2(\T).
\end{equation*}
Using the interpolation operator of \cite{Verfurth}, we can choose a $p\in P^2(\T)$ and an absolute constant such that
 \begin{equation*}
\|\Theta_\alpha-\Theta^I_\alpha\|_{3, \T}\le C\|\Theta_\alpha-p\|_{3, \T}\le C|\Theta_\alpha|_{3, \T}
\end{equation*}
Scale this estimate from $\T$ to $\tau$, we obtain the first estimate in \eqref{scaled-estimate}.

If $\tau$ has no edge on the free boundary $\E^F_h$, the second inequality in \eqref{scaled-estimate} is proved
in the same way. 
If $\tau$ has one or two edges on the free boundary, in place of the estimate
\eqref{ThetaI-bound}, we have that there is a $C$ depending only on the shape regularity of $\T$ such that
\begin{equation*} 
\|U^I_\alpha\|_{0,\T}\le C \|U_\alpha\|_{1,\T}\ \ \forall\ U_\alpha \in H^1(\T).
\end{equation*}
For any $p\in P^2(\T)$, we have
\begin{multline*}
\|U_\alpha-U^I_\alpha\|_{3, \T}\le \|U_\alpha-p\|_{3, \T}+\|(U_\alpha-p)^I\|_{3, \T}\\
\le \|U_\alpha-p\|_{3, \T}+C\|(U_\alpha-p)^I\|_{0, \T}
\le \|U_\alpha-p\|_{3, \T}+C\|U_\alpha-p\|_{1, \T}
\le C\|U_\alpha-p\|_{3, \T}.
\end{multline*}
Using the interpolation operator of \cite{Verfurth} again, we get
\begin{equation*}
\|U_\alpha-U^I_\alpha\|_{3, \T}\le C|U_\alpha|_{3,\T}.
\end{equation*}
Here $C$ only depends on the shape regularity of $\T$.
The second inequality in \eqref{scaled-estimate} then follows the scaling from $\T$ to $\tau$.
The third one is the same as the second one with  $\alpha=1$ or $2$.

Finally, we need to show that there exist interpolations $\M^{I\alpha\beta}$ and $\xi^{I\alpha}$ from continuous piecewise linear functions for $\M^{\eps\alpha\beta}$
and $\xi^{\eps\alpha}$, respectively,
such that
\begin{equation*}
\begin{gathered}
\sum_{\tau\in\T_h}\left(|\M^{\eps\alpha\beta}-\M^{I\alpha\beta}|^2_{0, \tau}+h^2_{\tau}|\M^{\eps\alpha\beta}-\M^{I\alpha\beta}|^2_{1, \tau}\right)
\le C
\sum_{\tau\in\T_h}h^4_{\tau}\|\M^{\eps\alpha\beta}\|^2_{2,\tau},\\
\sum_{\tau\in\T_h}\left(|\xi^{\eps\alpha}-\xi^{I\alpha}|^2_{0, \tau}+h^2_{\tau}|\xi^{\eps\alpha}-\xi^{I\alpha}|^2_{1, \tau}\right)
\le C
\sum_{\tau\in\T_h}h^4_{\tau}\|\xi^{\eps\alpha}\|^2_{2,\tau}.
\end{gathered}
\end{equation*}
This requirement can be 
met by choosing $\M^{I\alpha\beta}$ and $\xi^{I\alpha}$ as the nodal point interpolations  of $\M^{\eps\alpha\beta}$
and $\xi^{\eps\alpha}$, respectively.
In view of Theorem~\ref{N-fem-lem}, the proof is completed.
\end{proof}

\section{Higher order finite elements}

The finite element model \eqref{N-fem} can be defined on finite element spaces of higher order polynomials.
For integer $k>2$, we use
discontinuous piecewise degree $k$ polynomials to approximate the rotation components $\theta_\alpha$, use 
discontinuous piecewise degree $k$ polynomials with some modifications on elements that have edges on the free boundary $\E^F_h$
to approximate the displacement components $u_\alpha$ and $w$, 
and use 
continuous piecewise degree $k-1$ polynomials for components of the scaled membrane stress tensor $\M^{\alpha\beta}$ and 
components of the transverse shear stress vector $\xi^{\alpha}$.
In view of the interpolation requirements \eqref{uwI-edge1} and \eqref{uwI-edge2}, on an element $\tau$ that has one edge 
on the free boundary, we need to add $k$ polynomials to $P^k(\tau)$ for the variables $u_\alpha$ and $w$.
If an element $\tau$ has two edges on the free boundary, then we need to add $2k$ polynomials to $P^k(\tau)$ for the displacement 
components.

For such higher order finite element methods, we have the following theory.
If the Naghdi model solution has the regularity that $\theta\e_\alpha$,  $u\e_\alpha$,  and $w\e$ have finite norms in $H^{k+1}$, then
there is a constant $C$ that could be dependent on the shell mid surface, the Lam\'e coefficients of the elastic material, the polynomial degree $k$, 
and the shape regularity $\K$ of $\T_h$, but otherwise independent of the triangulation and the shell thickness $\eps$, such that
\begin{multline*}
\|(\!theta\e-\!theta^h, \!u\e-\!u^h, w\e-w^h)\|_{\H_h}\\
\le C
\left[1+\eps^{-1}
\max_{\tau\in\T_h}h^{k+1}_\tau\left(\sum_{\alpha,\beta,\lambda=1,2}|\Gamma^{\lambda}_{\alpha\beta}|_{k,\infty,\tau}+
\sum_{\alpha,\beta=1,2}|b_{\alpha\beta}|_{k,\infty,\tau}+\sum_{\alpha,\beta=1,2}|b^\beta_\alpha|_{k,\infty,\tau}\right)
\right]
\\
\left\{
\sum_{\tau\in\T_h}h^{2k}_{\tau}\left[\sum_{\alpha=1,2}\left(\|\theta\e_\alpha\|^2_{k+1,\tau}+\|u\e_\alpha\|^2_{k+1,\tau}\right)+\|w\e\|^2_{k+1,\tau}\right.\right.\\
\left.\left.+
\sum_{\alpha, \beta=1,2}\|\M^{\eps\alpha\beta}\|^2_{k,\tau}+\sum_{\alpha=1,2}\|\xi^{\eps\alpha}\|^2_{k,\tau}
\right]\right\}^{1/2}.
\end{multline*}
With $k$ being raised,  the locking effect is further reduced, and accuracy of the finite element approximation is enhanced.

\bibliographystyle{plain}

\appendix
\section*{Appendix}
\renewcommand{\theequation}{A.\arabic{equation}}
\renewcommand{\thesubsection}{A.\arabic{subsection}}
\setcounter{equation}{0}
\setcounter{subsection}{0}
\subsection{Consistency of the finite element model}
We verify that the solution $\!theta\e$, $\!u\e$, $w\e$, $\M\e$, $\!xi\e$ of the Naghdi model in mixed form \eqref{N-P-model} satisfies
the equation of the finite element model \eqref{N-fem} in which the test function $\!phi, \!v, z, \N, \!eta$ can be any piecewise
functions of sufficient regularity, not necessarily polynomials.
Under the assumption that  the shell material has constant
Lam\'e coefficients, we have $a^{\alpha\beta\gamma\delta}|_{\tau}=a_{\alpha\beta\gamma\delta}|_{\tau}=0$.
This is due to the fact that $a^{\alpha\beta}|_{\gamma}=a_{\alpha\beta}|_{\gamma}=0$.
On a region  $\t\tau\subset\t\Omega$, for any vectors $\theta_\alpha$, $\phi_\alpha$, $u_\alpha$, $v_\alpha$, and $\xi_\alpha$, scalars $w$ and $z$, 
and symmetric tensor $\M^{\alpha\beta}$,
the following identities follow from the Green's theorem \eqref{Green} directly.
\begin{multline}\label{id-1}
\int_{\t\tau}
a^{\alpha\beta\sigma\tau}\rho_{\sigma\tau}(\!theta, \!u, w)
\phi_{\alpha|\beta}=
-\int_{\t\tau}
a^{\gamma\beta\sigma\tau}\rho_{\sigma\tau|\beta}(\!theta, \!u, w)
\phi_{\alpha}+
\int_{\partial\t\tau}a^{\alpha\beta\sigma\tau}\rho_{\sigma\tau}(\!theta, \!u, w)
\phi_{\alpha}n_{\beta},\hfill
\end{multline}
\begin{multline*}
\int_{\t\tau}
a^{\alpha\beta\sigma\tau}\rho_{\sigma\tau}(\!theta, \!u, w)
b^{\gamma}_{\alpha}v_{\gamma|\beta}
=
-\int_{\t\tau}
a^{\alpha\beta\sigma\tau}[\rho_{\sigma\tau}(\!theta, \!u, w)
b^{\gamma}_{\alpha}]|_{\beta}v_{\gamma}
+\int_{\partial\t\tau}a^{\alpha\beta\sigma\tau}\rho_{\sigma\tau}(\!theta, \!u, w){n_{\beta}}
b^{\gamma}_{\alpha}v_{\gamma},\hfill
\end{multline*}
\begin{multline*}
\int_{\t\tau}a^{\alpha\beta\gamma\delta}\gamma_{\gamma\delta}(\!u, w)\frac12(v_{\alpha|\beta}+v_{\beta|\alpha})=
-\int_{\t\tau}a^{\alpha\beta\gamma\delta}\gamma_{\gamma\delta|\beta}(\!u, w)v_{\alpha}
+\int_{\partial\t\tau}a^{\alpha\beta\gamma\delta}\gamma_{\gamma\delta}(\!u, w){n_{\beta}}v_{\alpha},\hfill
\end{multline*}
\begin{multline*}
\int_{\t\tau}a^{\alpha\beta}\tau_\beta(\!theta, \!u, w)\partial_\alpha z=
-\int_{\t\tau}a^{\alpha\beta}\tau_{\beta|\alpha}(\!theta, \!u, w)z
+\int_{\partial\t\tau}a^{\alpha\beta}\tau_\beta(\!theta, \!u, w){n_{\alpha}} z,\hfill
\end{multline*}
\begin{multline*}
\int_{\t\tau}\M^{\alpha\beta}\frac12(v_{\alpha|\beta}+v_{\beta|\alpha})=
-\int_{\t\tau}\M^{\alpha\beta}|_{\beta}v_{\alpha}
+\int_{\partial\t\tau}\M^{\alpha\beta}{n_{\beta}}v_{\alpha},\hfill
\end{multline*}
\begin{multline*}
\int_{\t\tau}\xi^{\alpha}\partial_\alpha z=
-\int_{\t\tau}\xi^{\alpha}|_{\alpha}z
+\int_{\partial\t\tau}\xi^{\alpha}{n_{\alpha}} z.\hfill
\end{multline*}

Using these identities on $\t\Omega$ for several times,
we  write the Naghdi model \eqref{N-P-model} in the following mixed strong form.
\begin{multline}\label{N-P-strong}
\frac13\left[-a^{\alpha\beta\lambda\gamma}\rho_{\lambda\gamma|\beta}(\!theta, \!u, w)
+\kappa\mu a^{\alpha\beta}\tau_\beta(\!theta, \!u, w)\right]+\xi^\alpha=0 \text{ in }\Omega,\hfill
\end{multline}
\begin{multline*}
\frac13\left[a^{\delta\beta\lambda\gamma}[\rho_{\lambda\gamma}(\!theta, \!u, w)b^\alpha_\delta]|_\beta
-a^{\alpha\beta\lambda\gamma}\gamma_{\lambda\gamma|\beta}(\!u, w)
+\kappa\mu a^{\gamma\beta}\tau_\beta(\!theta, \!u, w)b^\alpha_\gamma\right]\\
-\M^{\alpha\beta}|_\beta+\xi^\gamma b^\alpha_\gamma=p^\alpha \text{ in }\Omega,
\end{multline*}
\begin{multline*}
\frac13\left[a^{\alpha\beta\lambda\gamma}\rho_{\lambda\gamma}(\!theta, \!u, w)c_{\alpha\beta}
-a^{\alpha\beta\lambda\gamma}\gamma_{\lambda\gamma}(\!u, w)b_{\alpha\beta}
-\kappa\mu a^{\alpha\beta}\tau_{\alpha|\beta}(\!theta, \!u, w)\right]\\
-\M^{\alpha\beta}b_{\alpha\beta}-\xi^\alpha|_\alpha=p^3 \text{ in }\Omega,
\end{multline*}
\begin{multline*}
\gamma_{\alpha\beta}(\!u, w)
-\eps^2a_{\alpha\beta\lambda\gamma}\M^{\lambda\gamma}=0,\quad
\tau_\alpha(\!theta, \!u, w)-
\eps^2\frac{1}{\kappa\mu}a_{\alpha\beta}\xi^{\eps\beta}=0 \text{ in }\Omega,\hfill
\end{multline*}
\begin{multline*}
\frac13a^{\alpha\beta\lambda\gamma}\rho_{\lambda\gamma}(\!theta, \!u, w)n_\beta=r^\alpha \text{ on }\partial^{S\cup F}\Omega,\hfill
\end{multline*}
\begin{multline*}
\frac13\left[-a^{\delta\beta\lambda\gamma}\rho_{\lambda\gamma}(\!theta, \!u, w)b^\alpha_\delta n_\beta
+a^{\alpha\beta\lambda\gamma}\gamma_{\lambda\gamma}(\!u, w)n_\beta\right]+\M^{\alpha\beta}n_\beta=q^\alpha \text{ on }\partial^{F}\Omega,\hfill
\end{multline*}
\begin{multline*}
\frac13\kappa\mu a^{\alpha\beta}\tau_{\beta}(\!theta, \!u, w)n_\alpha
+\xi^\alpha n_\alpha=q^3 \text{ on }\partial^{F}\Omega,\hfill
\end{multline*}
\begin{multline*}
u_\alpha=0,\ w=0\text{ on }\partial^{D\cup S}\Omega,\quad \theta_\alpha=0\text{ on }\partial^{D}\Omega.\hfill
\end{multline*}

For any piecewise vectors $\theta_\alpha$, $\phi_\alpha$, $u_\alpha$, $v_\alpha$, $\xi_\alpha$, and $\eta_\alpha$, scalars $w$ and $z$, and symmetric tensors
$\M^{\alpha\beta}$ and $\N^{\alpha\beta}$, on $\Omega_h$, summing up the bilinear forms
defined by \eqref{form_ub_a}, \eqref{form_b}, and \eqref{form_c}, we have

\begin{multline*}
\ub a(\!theta, \!u, w; \!phi, \!v, z)+b(\M, \!xi; \!phi, \!v, z)-b(\N, \!eta; \!theta, \!u, w)+\eps^2c(\M, \!xi; \N,\!eta)\hfill \\
=
\frac13\left\{\int_{\t\Omega_h}
\left[a^{\alpha\beta\lambda\gamma}\rho_{\lambda\gamma}(\!theta, \!u, w)
\rho_{\alpha\beta}(\!phi, \!v, z)
+
a^{\alpha\beta\lambda\gamma}\gamma_{\lambda\gamma}(\!u, w)
\gamma_{\alpha\beta}(\!v,z)\right.\right.\\
\left.+\kappa\mu a^{\alpha\beta}\tau_\beta(\!theta, \!u, w)\tau_\alpha(\!phi, \!v, z)\right]
\end{multline*}
\begin{multline*}
-\int_{\t\E^0_h}a^{\alpha\beta\lambda\gamma}\lbrac\rho_{\lambda\gamma}(\!phi, \!v, z)\rbrac\lbra\theta_\alpha\rbra_{n_\beta}
-\int_{\t\E^0_h}a^{\alpha\beta\lambda\gamma}\lbrac\rho_{\lambda\gamma}(\!theta, \!u, w)\rbrac\lbra\phi_\alpha\rbra_{n_\beta}\\
-\int_{\t\E^0_h}\kappa\mu a^{\alpha\beta}\lbrac\tau_\beta(\!theta, \!v, z)\rbrac\lbra w\rbra_{n_\alpha}
-\int_{\t\E^0_h}\kappa\mu a^{\alpha\beta}\lbrac\tau_\beta(\!theta, \!u, w)\rbrac\lbra z\rbra_{n_\alpha}
\end{multline*}
\begin{multline*}
+\int_{\t\E^0_h}\left[a^{\alpha\beta\lambda\gamma}\lbrac\rho_{\lambda\gamma}(\!phi, \!v, z)\rbrac b^\delta_\alpha-a^{\delta\beta\alpha\gamma}
\lbrac\gamma_{\alpha\gamma}(\!v, z)\rbrac
\right]
\lbra u_\delta\rbra_{n_\beta}\\
+\int_{\t\E^0_h}\left[a^{\alpha\beta\lambda\gamma}\lbrac\rho_{\lambda\gamma}(\!theta, \!u, w)\rbrac b^\delta_\alpha-a^{\delta\beta\alpha\gamma}
\lbrac\gamma_{\alpha\gamma}(\!u, w)\rbrac
\right]
\lbra v_\delta\rbra_{n_\beta}
\end{multline*}
\begin{multline*}
+\int_{\t\E^{D\cup S}_h}\left[a^{\alpha\beta\lambda\gamma}\rho_{\lambda\gamma}(\!phi, \!v, z)b^\delta_\alpha-a^{\delta\beta\alpha\gamma}
\gamma_{\alpha\gamma}(\!v, z)
\right]
u_\delta{n_\beta}\\
+\int_{\t\E^{D\cup S}_h}\left[a^{\alpha\beta\lambda\gamma}\rho_{\lambda\gamma}(\!theta, \!u, w)b^\delta_\alpha-a^{\delta\beta\alpha\gamma}
\gamma_{\alpha\gamma}(\!u, w)
\right]
v_\delta{n_\beta}\\
-\int_{\t\E^{D\cup S}_h}\kappa\mu a^{\alpha\beta}\tau_\beta(\!phi, \!v, z)w{n_\alpha}
-\int_{\t\E^{D\cup S}_h}\kappa\mu a^{\alpha\beta}\tau_\beta(\!theta, \!u, w)z{n_\alpha}\\
\left.
-\int_{\t\E^D_h}a^{\alpha\beta\lambda\gamma}\rho_{\lambda\gamma}(\!phi, \!v, z)\theta_\alpha{n_\beta}
-\int_{\t\E^D_h}a^{\alpha\beta\lambda\gamma}\rho_{\lambda\gamma}(\!theta, \!u, w)\phi_\alpha{n_\beta}\right\}
\end{multline*}
\begin{multline*}
+\int_{\t\Omega_h}\left[\M^{\alpha\beta}\gamma_{\alpha\beta}(\!v, z)+\xi^\alpha\tau_\alpha(\!phi, \!v, z)\right]\\
-\int_{\t\E^0_h}\left(\lbrac\M^{\alpha\beta}\rbrac\lbra v_{\alpha}\rbra_{n_{\beta}}+\lbrac\xi^\alpha\rbrac\lbra z\rbra_{n_\alpha}\right)
 -\int_{\t\E^{D\cup S}_h}\left(
\M^{\alpha\beta}{n_{\beta}}v_{\alpha}+\xi^\alpha n_\alpha z\right)
\end{multline*}
\begin{multline*}
-\int_{\t\Omega_h}\left[\N^{\alpha\beta}\gamma_{\alpha\beta}(\!u, w)+\eta^\alpha\tau_\alpha(\!theta, \!u, w)\right]\\
+\int_{\t\E^0_h}\left(\lbrac\N^{\alpha\beta}\rbrac\lbra u_{\alpha}\rbra_{n_{\beta}}+\lbrac\eta^\alpha\rbrac\lbra w\rbra_{n_\alpha}\right)
 +\int_{\t\E^{D\cup S}_h}\left(
\N^{\alpha\beta}{n_{\beta}}u_{\alpha}+\eta^\alpha n_\alpha w\right)\\
+\eps^2\int_{\t\Omega_h}\left(a_{\alpha\beta\gamma\delta}\M^{\gamma\delta}\N^{\alpha\beta}
+\frac{1}{\kappa\mu}a_{\alpha\beta}\xi^\alpha\eta^\beta\right).
\end{multline*}

Using the identities \eqref{id-1} on every element $\t\tau\in\t\T_h$, we rewrite the above form as
\begin{multline}\label{N-ub-a-weak}
\ub a(\!theta, \!u, w; \!phi, \!v, z)+b(\M, \!xi; \!phi, \!v, z)-b(\N, \!eta; \!theta, \!u, w)+\eps^2c(\M, \!xi; \N,\!eta)\\
=
\frac13\int_{\t\Omega_h}\left[-a^{\alpha\beta\lambda\gamma}\rho_{\lambda\gamma|\beta}(\!theta, \!u, w)
+\kappa\mu a^{\alpha\beta}\tau_\beta(\!theta, \!u, w)+3\xi^\alpha\right]\phi_\alpha\hfill
\end{multline}
\begin{multline*}
+\frac13\int_{\t\Omega_h}
\left\{a^{\gamma\beta\lambda\gamma}[\rho_{\lambda\gamma}(\!theta, \!u, w)b^\alpha_\gamma]|_\beta
-a^{\alpha\beta\lambda\gamma}\gamma_{\lambda\gamma|\beta}(\!u, w)
+\kappa\mu a^{\gamma\beta}\tau_\beta(\!theta, \!u, w)b^\alpha_\gamma\right.\\
\left.
-3\M^{\alpha\beta}|_\beta+3\xi^\gamma b^\alpha_\gamma\right\}v_\alpha
\end{multline*}
\begin{multline*}
+\frac13\int_{\t\Omega_h}
\left[a^{\alpha\beta\lambda\gamma}\rho_{\lambda\gamma}(\!theta, \!u, w)c_{\alpha\beta}
-a^{\alpha\beta\lambda\gamma}\gamma_{\lambda\gamma}(\!u, w)b_{\alpha\beta}
-\kappa\mu a^{\alpha\beta}\tau_{\alpha|\beta}(\!theta, \!u, w)\right.\\
\left.-3\M^{\alpha\beta}b_{\alpha\beta}-3\xi^\alpha|_\alpha\right]z
\end{multline*}
\begin{multline*}
+
\int_{\t\Omega_h}
\left[\gamma_{\alpha\beta}(\!u, w)
-\eps^2a_{\alpha\beta\lambda\gamma}\M^{\lambda\gamma}\right]\N^{\alpha\beta}
+\int_{\t\Omega_h}
\left[\tau_\alpha(\!theta, \!u, w)-
\eps^2\frac{1}{\kappa\mu}a_{\alpha\beta}\xi^{\beta}\right]\eta^\alpha\hfill
\end{multline*}
\begin{multline*}
-\frac13\int_{\t\E^0_h}a^{\alpha\beta\lambda\gamma}\lbrac\rho_{\lambda\gamma}(\!phi, \!v, z)\rbrac\lbra\theta_\alpha\rbra_{n_\beta}
+\frac13\int_{\t\E^0_h}a^{\alpha\beta\lambda\gamma}\lbra\rho_{\lambda\gamma}(\!theta, \!u, w)\rbra_{n_\beta}\lbrac\phi_\alpha\rbrac\\
-\frac13\int_{\t\E^0_h}\kappa\mu a^{\alpha\beta}\lbrac\tau_\beta(\!theta, \!v, z)\rbrac\lbra w\rbra_{n_\alpha}
+\frac13\int_{\t\E^0_h}\kappa\mu a^{\alpha\beta}\lbra\tau_\beta(\!theta, \!u, w)\rbra_{n_\alpha}\lbrac z\rbrac
\end{multline*}
\begin{multline*}
+\frac13\int_{\t\E^0_h}\left[a^{\alpha\beta\lambda\gamma}\lbrac\rho_{\lambda\gamma}(\!phi, \!v, z)\rbrac b^\delta_\alpha-a^{\delta\beta\alpha\gamma}
\lbrac\gamma_{\alpha\gamma}(\!v, z)\rbrac
\right]
\lbra u_\delta\rbra_{n_\beta}\\
-\frac13\int_{\t\E^0_h}\left[a^{\alpha\beta\lambda\gamma}\lbra\rho_{\lambda\gamma}(\!theta, \!u, w)\rbra_{n_\beta} b^\delta_\alpha-a^{\delta\beta\alpha\gamma}
\lbra\gamma_{\alpha\gamma}(\!u, w)\rbra_{n_\beta}
\right]
\lbrac v_\delta\rbrac
\end{multline*}
\begin{multline*}
+\frac13\int_{\t\E^{D\cup S}_h}\left[a^{\alpha\beta\lambda\gamma}\rho_{\lambda\gamma}(\!phi, \!v, z)b^\delta_\alpha-a^{\delta\beta\alpha\gamma}
\gamma_{\alpha\gamma}(\!v, z)
\right]
u_\delta{n_\beta}\\
-\frac13\int_{\t\E^{F}_h}\left[a^{\alpha\beta\lambda\gamma}\rho_{\lambda\gamma}(\!theta, \!u, w)b^\delta_\alpha-a^{\delta\beta\alpha\gamma}
\gamma_{\alpha\gamma}(\!u, w)
\right]
v_\delta{n_\beta}\\
-\frac13\int_{\t\E^{D\cup S}_h}\kappa\mu a^{\alpha\beta}\tau_\beta(\!phi, \!v, z)w{n_\alpha}
+\frac13\int_{\t\E^{F}_h}\kappa\mu a^{\alpha\beta}\tau_\beta(\!theta, \!u, w)z{n_\alpha}\\
-\frac13\int_{\t\E^D_h}a^{\alpha\beta\lambda\gamma}\rho_{\lambda\gamma}(\!phi, \!v, z)\theta_\alpha{n_\beta}
+\frac13\int_{\t\E^{S\cup F}_h}a^{\alpha\beta\lambda\gamma}\rho_{\lambda\gamma}(\!theta, \!u, w)\phi_\alpha{n_\beta}
\end{multline*}
\begin{multline*}
+\int_{\t\E^0_h}\left(\lbra\M^{\alpha\beta}\rbra_{n_\beta}\lbrac v_{\alpha}\rbrac+\lbra\xi^\alpha\rbra_{n_\alpha}\lbrac z\rbrac\right)
 +\int_{\t\E^{F}_h}\left(
\M^{\alpha\beta}{n_{\beta}}v_{\alpha}+\xi^\alpha n_\alpha z\right)\\
+\int_{\t\E^0_h}\left(\lbrac\N^{\alpha\beta}\rbrac\lbra u_{\alpha}\rbra_{n_{\beta}}+\lbrac\eta^\alpha\rbrac\lbra w\rbra_{n_\alpha}\right)
 +\int_{\t\E^{D\cup S}_h}\left(
\N^{\alpha\beta}{n_{\beta}}u_{\alpha}+\eta^\alpha n_\alpha w\right).
\end{multline*}
Since there is no jump in the Naghdi model solution $\!theta\e,\!u\e, w\e, \M\e, \!xi\e$, in view of the definitions \eqref{form_a} and \eqref{form_ub_a}
we have $a(\!theta\e, \!u\e, w\e; \!phi, \!v, z)=\ub a(\!theta\e, \!u\e, w\e; \!phi, \!v, z)$ for any piecewise $(\!phi, \!v, z)$.
We therefore have that for any piecewise regular test function $(\!phi, \!v, z)$
\begin{multline*}
a(\!theta\e, \!u\e, w\e; \!phi, \!v, z)+b(\M\e, \!xi\e; \!phi, \!v, z)-b(\N, \!eta; \!theta\e, \!u\e, w\e)+\eps^2c(\M\e, \!xi\e;  \N, \!eta)\\
=
\ub a(\!theta\e, \!u\e, w\e; \!phi, \!v, z)+b(\M\e, \!xi\e; \!phi, \!v, z)-b(\N, \!eta; \!theta\e, \!u\e, w\e)+\eps^2c(\M\e, \!xi\e;  \N, \!eta)
\\
=\langle\!f; \!phi, \!v, z\rangle.
\end{multline*}
The linear form $\langle\!f; \!phi, \!v, z\rangle$ is defined by \eqref{form_f}. The last equation follows from comparing
\eqref{N-ub-a-weak} with the strong form of the Naghdi model \eqref{N-P-strong}.
This verifies the consistency of the finite element model
\eqref{N-fem} with the Naghdi shell model \eqref{N-P-model}.

\subsection{Proofs of Theorems~\ref{N-primary-limit} and Theorem~\ref{isomorphism-thm}}
For completeness, we include the proofs for results on the asymptotic analysis in the in the abstract setting.
\begin{proof}[Proof of Theorems~\ref{N-primary-limit}]
In view of equation \eqref{limitas}, we have $\langle f,v\rangle-(Au^0, Av)_U=0$ $\forall v\in\ker B$. Thus there is a unique $\zeta\in W^*$ such that
$\langle f,v\rangle-(Au^0, Av)_U=\langle\zeta, Bv\rangle$.
Subtracting $(Au^0, Av)$ from both sides of the equation \eqref{prob1as} and using the fact that $Bu^0=0$, we have
\begin{equation*}
\eps^2(A(u\e-u^0),Av)_U+(B(u\e-u^0),Bv)_{\overline W}=\eps^2\langle\zeta, Bv\rangle
\quad \forall\ v\in H.
\end{equation*}
This problem is in the form that was analyzed in  \cite{Caillerie} and \cite{CR1}. By
Theorem 2.1 of \cite{CR1},  we have
\begin{equation*}
\|A(u\e-u^0)\|_U+\eps^{-1}\|Bu\e\|_{V}
+\|\eps^{-2}\pi_{\overline W}Bu\e-\zeta\|_{W^*}
\simeq\|\zeta\|_{W^*+\eps(\overline W)^*},\\
\end{equation*}
The conclusion of the theorem them follows from the fact that
$\lim_{\eps\to 0}\|\zeta\|_{W^*+\eps(\overline W)^*}=0$
\cite{BL}, and
$\|\eps^{-2}\pi_{\overline W}Bu\e-\zeta\|_{W^*}
=\|\eps^{-2}Bu\e-\M\|_{\overline{\overline W}}$. Here $\M=j_{[W^*\to\overline{\overline W}]}\zeta$.
\end{proof}

\begin{proof}[Proof of Theorem~\ref{isomorphism-thm}]
The second inequality is obvious.
To prove the first inequality, we first assume that $W$ is dense in $V$. 
Then  $V^*$ is dense in $W^*$, and  we have $(\eps V^*\cap W^*)^*=\eps^{-1}V+W$ \cite{BL}.
We also see from the definition \eqref{isomorphism-weak} that $|q|_{\overline V}=\|\pi_Vq\|_{W^*}$. Here $\pi_V:V\to V^*$ is the inverse of Riesz representation.
For any $(u, p)\in H\x V$, we let $(f, g)\in H^*\x V^*$ be the corresponding right hand side functional in
the equation \eqref{isomorphism-abs}. We write $u=u_1+u_2$ and $p=p_1+p_2$,
with $(u_1, p_1)$ solving
\begin{equation*} 
\begin{gathered}
a(u_1, v)+b(v, p_1)=\langle f, v\rangle\ \ \forall\ v\in H,\\
b(u_1, q)-\eps^2c(p_1, q)=0\ \ \forall\ q\in V,
\end{gathered}
\end{equation*}
while $(u_2, p_2)$ solves
\begin{equation*} 
\begin{gathered}
a(u_2, v)+b(v, p_2)=0\ \ \forall\ v\in H,\\
b(u_2, q)-\eps^2c(p_2, q)=\langle g, q\rangle\ \ \forall\ q\in V.
\end{gathered}
\end{equation*}
It can be shown that  $\|u_1\|_H+|p_1|_{\overline V}+\eps\|p_1\|_V\le C\|f\|_{H^*}$ \cite{ABrezzi2}. This is to say
\begin{equation}\label{iso-est1}
\|u_1\|_H+|p_1|_{\overline V}+\eps\|p_1\|_V\le C
\sup_{v\in H}\frac{|a(u, v)+b(v, p)|}{\|v\|_{H}}
\end{equation}
From the first equation of the above second system, we see that $|p_2|_{\overline V}\le \|u_2\|_H$.
Let $u_0\in H$ be an element such that $\|u_0\|_H=\|Bu_0\|_W$. We write
\begin{equation*}
\langle g, q\rangle=(i_V g, q)=(Bu_0, q)+(i_V g_1, q)=b(u_0, q)+\langle g_1, q\rangle.
\end{equation*}
We thus have
\begin{equation*} 
\begin{gathered}
a(u_2-u_0, v)+b(v, p_2)=-a(u_0, v)\ \ \forall\ v\in U,\\
b(u_2-u_0, q)-\eps^2c(p_2, q)=\langle g_1, q\rangle\ \ \forall\ q\in V.
\end{gathered}
\end{equation*}
Taking $v=u_2-u_0$ and $q=p_2$ in this equation, and sum, we have
\begin{equation*}
\|u_2-u_0\|_H^2+\eps^2\|p_2\|^2_V=-a(u_0, u_2-u_0)-(\eps^{-1}i_V g_1, \eps p_2).
\end{equation*}
Using Cauchy--Schwarz inequality, we get
\begin{equation*}
\|u_2-u_0\|_H+\eps\|p_2\|_V\le C \|u_0\|_H+\eps^{-1}\|i_V g_1\|_V.
\end{equation*}
Thus
\begin{equation*}
\|u_2\|_H+|p_2|_{\overline V}+\eps\|p_2\|_V\le C(\|u_0\|_H+\eps^{-1}\|i_V g_1\|_V)\le C
(\|Bu_0\|_W+\eps^{-1}\|i_V g_1\|_V).
\end{equation*}
Since this is valid for any decomposition of $g$, we get
\begin{multline}\label{iso-est2}
\|u_2\|_H+|p_2|_{\overline V}+\eps\|p_2\|_V\le C\|i_Vg\|_{W+\eps^{-1}V}=\|i_V g\|_{(W^*\cap\eps V^*)^*}\\=
\sup_{q\in V}\frac{\langle i _Vg, \pi_Vq\rangle}{\|\pi_Vq\|_{W^*}+\eps\|\pi_Vq\|_{V^*}}=
\sup_{q\in V}\frac{\langle g, q\rangle}{|q|_{\overline V}+\eps\|q\|_{V}}.
\end{multline}
It follows from \eqref{iso-est1} and \eqref{iso-est2} that when $W$ is dense in $V$ for any $(u, p)\in H\x V$ we have
\begin{equation}\label{iso-est3}
\|u\|_H+|p|_{\overline V}+\eps\|p\|_V\le C
\sup_{(v, q)\in H\x V}\frac{a(u, v)+b(v, p)
-b(u, q)+\eps^2c(p, q)}
{\|v\|_H+|q|_{\overline V}+\eps\|q\|_{V}}.
\end{equation}

If $W$ is not dense in $V$, we let $\overline W$ be the closure of $W$ in $V$, and
decompose $V$ orthogonally as $V=\overline W\oplus\overline W^\perp$. Then any $q\in V$ can be written as
$q=q_W+q_\perp$ such that $q_W\in\overline W$ and $q_\perp\in\overline W^\perp$. We have $b(v, q)=b(v, q_W)$ and $|q|_{\overline V}=|q_W|_{\overline V}$
and $\|q\|^2_V=\|q_W\|^2_V+\|q_\perp\|^2_V$.
For any $u\in H$ and $p\in V$, it follows from \eqref{iso-est3} that
\begin{equation*}
\|u\|_H+|p_W|_{\overline V}+\eps\|p_W\|_V
\le C
\sup_{v\in H, q_W\in \overline W}\frac{a(u, v)+b(v, p_W)-
b(u, q_W)+\eps^2c(p_W, q_W)}{\|v\|_H+|q_W|_{\overline V}+\eps\|q_W\|_V}.
\end{equation*}
Thus
\begin{multline*}
\|u\|_H+|p|_{\overline V}+\eps\|p\|_V\le C
(\|u\|_H+|p_W|_{\overline V}+\eps\|p_W\|_V+\eps\|p_\perp\|_V)\\
\le C
\sup_{v\in H, q_W\in \overline W}\frac{a(u, v)+b(v, p_W)-
b(u, q_W)+\eps^2c(p_W, q_W)}{\|v\|_H+|q_W|_{\overline V}+\eps\|q_W\|_V}+C\eps\|p_\perp\|_V\\
\le C
\sup_{v\in H, q_W\in\overline W, q_\perp\in \overline W^\perp}\frac{a(u, v)+b(v, p_W)-
b(u, q_W)+\eps^2c(p_W, q_W)+\eps^2c(p_\perp, q_\perp)}{\|v\|_H+|q_W|_{\overline V}+\eps\|q_W\|_V+\eps\|q_\perp\|_V}\\
\le C\sup_{v\in H, q\in V}\frac{a(u, v)+b(v, p)-
b(u, q)+\eps^2c(p, q)}{\|v\|_H+|q|_{\overline V}+\eps\|q\|_V}.
\end{multline*}
\end{proof}

\end{document}